\documentclass[a4paper,11pt]{article}

\input{styleart.sty}
\date{}

\author{Chlo\'e Perin and Rizos Sklinos}
\title{Forking and JSJ decompositions in the free group}

\begin{document}

\maketitle

\begin{abstract} 
We give a description of the model theoretic relation of forking independence in terms of the notion of JSJ decompositions 
in non abelian free groups.  
\end{abstract}

\section{Introduction}
In this paper we examine the model theoretic notion of forking independence in non abelian free groups. 

Forking independence was introduced by Shelah as part of the machinery needed in his classification program (see Section \ref{Stab}). 
It is an abstract independence relation between two tuples of a structure over a set of parameters. 

In the tame context 
of stable first-order structures, forking independence possesses certain nice properties (see Fact \ref{forkprop}). 
As a matter of fact the existence of an independence relation having these properties characterizes stable theories 
and the relation in this case must be exactly forking independence. So it is not surprising that   
forking independence has grown to have its own ontology and many useful notions 
have been introduced around it, first inside the context of stable theories and then adapted and developed more generally.
 
In some stable algebraic structures, such as 
a module or an algebraically closed field, forking independence admits an algebraic 
interpretation. In the latter case this is easily described: if $\bar{b},\bar{c}$ are finite tuples in an 
algebraically closed field $\mathcal{K}$ and $L$ is a subfield, then $\bar{b}$ is independent from $\bar{c}$ 
over $L$ if and only if the transcendence degree of $L(\bar{b}\bar{c})$ over $L(\bar{c})$ 
is the same as the transcendence degree of $L(\bar{b})$ over $L$. For a description in the case of modules we refer the reader to \cite{StevMod}\cite{PillMod}.

Philosophically speaking, in every ``natural'' stable structure one should be able 
to understand forking independence in terms of the underlying geometric or algebraic nature. 
Sela \cite{SelaStability} proved that (non-cyclic) torsion-free hyperbolic groups are stable, 
thus it is natural to ask whether the forking independence relation can be given an algebraic interpretation in these groups. 
This paper, following this line of thought, gives such an interpretation in free groups and in some torsion-free hyperbolic 
groups in terms of the Grushko and JSJ decompositions. 

The first main result of this paper is:

\begin{thmIntro} \label{FreeFactorsIntro}
Let $\bar{b},\bar{c}$ be tuples of elements in the free group $\F_n$ and let $A$ be a free factor of $\F_n$. Then 
$\bar{b}$ and $\bar{c}$ are independent over $A$ if and only if $\F_n$ admits a free decomposition $\F_n=\F*A*\F'$ with 
$\bar{b}\in \F*A$ and $\bar{c}\in A*\F'$.
\end{thmIntro}

Thus two finite tuples are independent over $A$ if and only if they live in "essentially disjoint" parts of the Grushko decomposition 
of $\F_n$ relative to $A$ (i.e. the maximal decomposition of $\F_n$ as a free product in which $A$ is contained in one of the factors). 
The essential ingredients of the proof of our first result is the homogeneity of non abelian free groups and a result of independent 
interest concerning the stationarity of types in the theory of non abelian free groups (see Theorem \ref{stat}).

The relative Grushko decomposition of a group with respect to a set of parameters is a way to see all the splittings of the group as a 
free product in which the set of parameters is contained in one of the factors. The relative cyclic JSJ decomposition is a 
generalization of this: it is a graph of groups decomposition which encodes all the splittings of the group as an amalgamated 
product or an HNN extension over a cyclic group, for which the parameter set is contained in one of the factors 
(see Section \ref{JSJSec}).

The second result deals with the case where the parameter set is not contained in any proper free factor, 
so that the relative Grushko decomposition is trivial, and tells us that two tuples are then independent over $A$ if and only if 
they live in "essentially disjoint" parts of the cyclic JSJ decomposition of $\F_n$ relative to $A$. 

\begin{thmIntro} \label{FreelyIndecIntro} 
 Let $\F_n$ be freely indecomposable with respect to $A$.

Let $(\Lambda, v_A)$ be the pointed cyclic JSJ decomposition of $\F_n$ with respect to $A$. Let $\bar{b}$ and $\bar{c}$ be tuples in $\F_n$, 
and denote by $\Lambda_{A\bar{b}}$ (respectively $\Lambda_{A\bar{c}}$) the minimal subgraphs of groups of 
$\Lambda$ whose fundamental group contains the subgroups $\langle A, \bar{b}\rangle$ (respectively $\langle A, \bar{c} \rangle$) of $\F_n$.

Then $\bar{b}$ and $\bar{c}$ are independent over $A$ if and only if each connected component of 
$\Lambda_{A\bar{b}} \cap \Lambda_{A\bar{c}}$ contains at most one non Z-type vertex, 
and such a vertex is of non surface type.
\end{thmIntro} 

This is a special case of Theorem \ref{FreelyIndec}, where we prove it for torsion-free 
hyperbolic groups which are concrete over a set of parameters $A$. A group is said to be concrete with respect 
to the set of parameters $A$ if it is freely indecomposable with respect to $A$, and does not 
admit an extended hyperbolic floor structure over $A$ - that is, $A$ is not contained in a proper 
retract of $G$ which satisfies certain properties, see \cite{LouderPerinSklinosTowers}. 
  
In this second result, the middle step between the purely model theoretic notion of forking independence and the 
purely geometric one of JSJ decomposition is that of understanding the automorphism group of $\F_n$ relative to $A$. 
Indeed, the JSJ decomposition enables us to give in this setting a very good description (up to a finite index) 
of the group of automorphisms which fix $A$ pointwise. On the other 
hand, in many cases the model theoretic definitions bear a 
strong relation to properties of invariance under automorphisms (as can be seen in Section \ref{Stab}).  

It is remarkable that once more the tools of geometric group theory prove so adequate to understand the 
first-order theory of torsion-free hyperbolic groups. The proof of Theorem \ref{FreelyIndecIntro} makes use of 
recent deep results of this field, such as Masur and Minsky's results on the curve complex of a surface. 
 
The paper is organised as follows. Section \ref{Stab} gives a thorough exposition of the model theoretic notions needed, 
and the necessary background on forking independence. Section \ref{FreeFactorsSec} is devoted to the proof of Theorem \ref{FreeFactorsIntro}. 
Section \ref{JSJSec} gives relations between the automorphisms and JSJ decomposition. In Section \ref{IsolatingSec} 
we prove that torsion-free hyperbolic groups are atomic over sets of parameters with respect to which they are concrete. 
Section \ref{AlgClosureSec} recalls some properties of algebraic closures, 
and a description of algebraic closures in torsion-free hyperbolic groups. Section \ref{CCSec} gives a brief 
introduction to the curve complex and states the results needed in Section \ref{FreelyIndecSec}, which is 
devoted to the proof of Theorem \ref{FreelyIndecIntro}. In the final section we give some examples and 
make some further remarks on our results.

\paragraph{Acknowledgements.}
We wish to thank Anand Pillay for suggesting the problem.  
Part of this work was conducted while the second named author was a PhD student under the 
supervision of A. Pillay and he would like to thank him for teaching him how to fork. 

We are grateful to Zlil Sela for a number of helpful suggestions over the course of this work. 

We would also like to thank Udi Hrushovski for the idea of Theorem \ref{stat}.

Finally, this work would have not been made possible without the support of 
A.Pillay and Z.Sela that allowed the authors to visit each other and collaborate.

\section{An introduction to forking independence} \label{Stab}
In this section we give an almost complete account of the model theoretic background needed for the rest of the paper. 

\subsection{Basic Notions}

We first recall and fix notations for the basic notions of model theory. Let $\mathcal{M}$ be a structure, and  
$\mathcal{T}h(\mathcal{M})$ its complete first-order theory (namely, the set of all first-order sentences that hold in $\mathcal{M}$). 

An {\em $n$-type} $p(\bar{x})$ of $\mathcal{T}h(\mathcal{M})$ is a set of formulas without parameters in $n$ variables which is 
consistent with $\mathcal{T}h(\mathcal{M})$. 
A type $p(\bar{x})$ is called complete if for every $\phi(\bar{x})$ either $\phi$ or $\lnot\phi$ is in $p(\bar{x})$. 
For example if $\bar{a}$ is a tuple in $\mathcal{M}$, the set $\tp^{\mathcal M}(\bar{a})$ of all formulas satisfied by $\bar{a}$ is a complete type.

If $A\subset\mathcal{M}$ is a set of parameters, we denote by $S_n^{\mathcal{M}}(A)$ the set of all 
complete $n$-types of $\mathcal{T}h((\mathcal{M},\{a\}_{a\in A}))$ 
(where $(\mathcal{M},\{a\}_{a\in A})$ is the structure obtained from $\mathcal{M}$ by adding the elements of $A$ as constant symbols). 
We also note that the set of $n$-types over the empty set of a first-order theory $T$ is usually denoted by $S_n(T)$.
 
It is easy to see that $S_n^{\mathcal{M}}(A)$ is a Stone space when equipped with the topology defined 
by the following basis of open sets $[\phi(\bar{x})]=\{p\in S_n^{\mathcal{M}}(A) : \phi\in p\}$, where $\phi(\bar{x})$ is a formula 
with parameters in $A$. 
A type $p\in S_n^{\mathcal{M}}(A)$ is {\em isolated} if there is a formula, $\phi\in p$,  
such that $[\phi]=\{p\}$. 

We continue with some basic definitions.


\begin{defi}
Let $A$ be a subset of a structure $\mathcal{M}$. Then $\mathcal{M}$ is called atomic over $A$ 
if every type in $S_n^{\mathcal{M}}(A)$ which is realized in $\mathcal{M}$ is isolated. 
\end{defi}

Equivalently, if $\mathcal{M}$ is a countable structure, then $\mathcal{M}$ is atomic if the orbit $Aut(\mathcal{M}).\bar{b}$ 
of every finite tuple $\bar{b}\in \mathcal{M}$ is $\emptyset$-definable. If the orbits are merely $\emptyset$-type-definable (i.e.
the solution set of a type in $\mathcal{M}$), then we obtain the notion of homogeneity.



\begin{defi}
A countable structure $\mathcal{M}$ is said to be homogeneous, if for every two finite tuples $\bar{b},\bar{c}\in \mathcal{M}$ 
such that $tp^{\mathcal{M}}(\bar{b})=tp^{\mathcal{M}}(\bar{c})$ there is an automorphism of 
$\mathcal{M}$ sending $\bar{b}$ to $\bar{c}$.
\end{defi}





Let $\bar{a},A\subset \mathcal{M}$. We say $\bar{a}$  
is {\em algebraic} (respectively {\em definable}) over $A$, if there is a formula $\phi(\bar{x})$ with parameters in $A$ 
such that $\bar{a}\in \phi(\mathcal{M})$ and $\phi(\mathcal{M})$ is finite (respectively has cardinality one). We denote the set of 
algebraic (respectively definable) tuples over $A$ by $acl_{\mathcal{M}}(A)$ (respectively $dcl_{\mathcal{M}}(A)$). The following lemma is immediate.

\begin{lemma}\label{AlgClos}
Let $\mathcal{M}$ be countable and atomic over $A$. Then for any $\bar{a}\in \mathcal{M}$, 
$\bar{a}$ is algebraic (respectively definable) over $A$ if and only if 
$\abs{\{f(\bar{a}) : f\in Aut(\mathcal{M}/A)\}}$ is finite (respectively has cardinality one).  
\end{lemma}
\begin{proof}
Just note that the orbit of any tuple under $\Aut(\mathcal{M}/A)$ is a definable set over $A$.
\end{proof}
 


      

\subsection{Stability Theory}
Stability theory is an important part of modern model theory. The rudiments 
of stability can be found in the seminal work of Morley proving \L{}os conjecture, namely that a countable theory 
is categorical in an uncountable cardinal if and only if it is categorical in all uncountable cardinals. 
A significant aspect of Morley's work is that he assigned an invariant 
(a dimension for some independence relation) to a model 
that determined the model up to isomorphism. In full 
generality most of the results concerning stability are attributed to Shelah \cite{Shelah}. Shelah 
has established several dividing lines separating ``well behaved'' theories from 
theories which do not have a structure theorem classifying their models. One such dividing line 
is stable versus unstable, where if a first-order theory $T$ is unstable then it has 
the maximum number of models, $2^{\kappa}$, for each cardinal $\kappa\geq 2^{\abs{T}}$.

A first-order theory $T$ is {\em stable} if it prevents the definability of an infinite linear order. 
More formally we have: a first-order formula $\phi(\bar{x},\bar{y})$ in a structure $\mathcal{M}$ 
has the {\em order property} if there are sequences $(\bar{a}_n)_{n<\omega},(\bar{b}_n)_{n<\omega}$ 
such that $\mathcal{M}\models\phi(\bar{a}_n,\bar{b}_m)$ if and only if $m<n$. 

\begin{defi}
A first-order theory $T$ is stable if no formula has the order property in a model of $T$. 
\end{defi}
    
By the discussion above it is apparent that the development of an abstract independence relation 
enabling us to assign a dimension to several sets will be useful. This is what brought Shelah to define forking independence. 

The rest of the subsection is devoted to a thorough description of forking independence in stable theories. Unless otherwise stated, 
all the results in this subsection which are stated without a proof can be found in \cite[Chapter 1, Sections 1-2]{PillayStability}.
 
We fix a stable first-order theory $T$ and we work in a ``big'' saturated model $\mathbb{M}$ of $T$, which  
is usually called the {\em monster model} (see \cite[p.218]{MarkerModelTheory}).

We write $tp(\bar{a}/A)$ for $tp^{\mathbb{M}}(\bar{a}/A)$ and 
$S_n(A)$ for $S_n^{\mathbb{M}}(A)$. 

\begin{defi} A formula {\em $\phi(\bar{x},\bar{b})$ forks over $A$} if there are $n<\omega$ and 
an infinite sequence $(\bar{b}_i)_{i<\omega}$ such that $tp(\bar{b}/A)=tp(\bar{b_i}/A)$ for $i<\omega$, and 
the set $\{\phi(\bar{x},\bar{b}_i) : i<\omega\}$ is $n$-inconsistent.

A tuple $\bar{a}$ is {\em independent} from $B$ over $A$ (denoted $\bar{a} \underset{A}{\forkindep} B$)
if there is no formula in $tp(\bar{a}/B)$ which forks over $A$.
\end{defi}

In Section 2 of \cite{LouderPerinSklinosTowers} we give 
an intuitive account of the above definition. 



The following observation is immediate.

\begin{rmk}\label{ForkImpl}
Let $\mathbb{M}\models \phi(\bar{x})\rightarrow\psi(\bar{x})$ and $\psi(\bar{x})$ forks over $A$. Then $\phi(\bar{x})$ forks over $A$.
\end{rmk}

\begin{defi} If $p\in S_n(A)$ and $A\subseteq B$, 
then $q:=tp(\bar{a}/B)$ is called a {\em non forking extension of $p$}, if $p\subseteq q$ 
and moreover $\bar{a} \underset{A}{\forkindep} B$.
\end{defi}

\begin{defi} A type $p\in S_n(A)$ is called {\em stationary} 
if for any $B\supseteq A$, $p$ has a unique non-forking extension over $B$.  
\end{defi}

\begin{defi} Let $C=\{\bar{c}_i : i\in I\}$ be a set of tuples, we say that {\em $C$ is an independent set over $A$} if for 
every $i\in I$, $\bar{c}_i \underset{A}{\forkindep} \bigcup C\setminus\{\bar{c}_i\}$.
\end{defi}

If $p$ is a type over $A$ which is stationary and  $(a_i)_{i<\kappa}, (b_i)_{i<\kappa}$ are both independent sets over $A$ of 
realizations of $p$, 
then $tp((a_i)_{i<\kappa}/A)=tp((b_i)_{i<\kappa}/A)$. 
This allows us to denote by $p^{(\kappa)}$ the type of $\kappa$-independent realizations of $p$. Is not 
hard to see that if $p$ is stationary then so is $p^{(\kappa)}$.

We observe the following behavior of forking independence inside a countable atomic model of $T$. 

\begin{lemma}\label{AtoFork}
Let $\mathcal{M}\models T$. Let $\bar{b},A\subset \mathcal{M}$, 
and $\mathcal{M}$ be countable and atomic over $A$. Suppose $X:=\phi(\mathcal{M},\bar{b})$ contains 
a non empty almost $A$-invariant subset (i.e. a subset that has finitely many images under $Aut(\mathcal{M}/A)$). 
Then $\phi(\bar{x},\bar{b})$ does not fork over $A$.   
\end{lemma}
\begin{proof}
Suppose it does, then there is an infinite sequence, $(\bar{b}_i)_{i<\omega}$ in $\mathbb{M}$ such that 
$tp(\bar{b}_i/A)=tp(\bar{b}/A)$ and $\{\phi(\bar{x},\bar{b}_i): i<\omega\}$ is $k$-inconsistent for some $k<\omega$. 
Since $\mathcal{M}$ is atomic over $A$, we have that $tp(\bar{b}/A)$ is isolated, say by $\psi(\bar{y})$. Thus, 
for arbitrarily large $\lambda$ the following sentence (over $A$) is true:
$$\mathbb{M}\models \exists \bar{y}_1,\ldots,\bar{y}_{\lambda}[(\psi(\bar{y}_1)\land\ldots
\land\psi(\bar{y}_{\lambda}))\bigwedge \textrm{``any $k$-subset of 
$\{\phi(\bar{x},\bar{y}_1),\ldots,\phi(\bar{x},\bar{y}_{\lambda})\}$ is inconsistent''}]$$

But the above sentence is true in $\mathcal{M}$, and this contradicts the hypothesis that 
$\phi(\bar{x},\bar{b})$ contains a non empty almost $A$-invariant set.

\end{proof}

We list some properties of forking independence (recall that we assumed the theory $T$ to be stable).

\begin{fact}\label{forkprop}
 \begin{itemize}
  \item[(i)](existence of non-forking extensions) Let $p\in S_n(A)$ and $A\subseteq B$. 
  Then there is a non forking extension of $p$ over $B$;
 \item[(ii)](symmetry) $\bar{a} \underset{A}{\forkindep} \bar{b}$ if and only if 
 $\bar{b} \underset{A}{\forkindep} \bar{a}$;
 \item[(iii)](local character) For any $\bar{a},A$, there is $A'\subseteq A$ with $\abs{A'}\leq \abs{T}$, such that 
 $\bar{a} \underset{A'}{\forkindep} A$;
 \item[(iv)](transitivity) Let $A\subseteq B\subseteq C$. Then $\bar{a} \underset{A}{\forkindep} C$    
if and only if $\bar{a} \underset{A}{\forkindep} B$ and $\bar{a} \underset{B}{\forkindep} C$;
 \item[(v)](boundedness) Every type over a model is stationary.
 \end{itemize}
\end{fact}

In fact the above properties of forking independence characterize stable theories in the sense that if a theory 
$T$ admits a sufficiently saturated model on which we can define an independence relation on triples of sets satisfying $(i)-(v)$, 
then $T$ is stable and the relation is exactly forking independence.  

The following lemma is useful in practice.

\begin{lemma}\label{ForkAlg}
 Let $A\subseteq B$. Then $\bar{a} \underset{A}{\forkindep} B$ if and only if $acl(\bar{a}A) \underset{acl(A)}{\forkindep} acl(B)$. 
\end{lemma}

The following theorem answers the question of how much ``information'' a type should include in order to be stationary.

\begin{thm}[Finite Equivalence Relation Theorem]
Let $p_1,p_2\in S_n(B)$ be two distinct types, let $A\subseteq B$ and suppose that $p_1, p_2$ 
both do not fork over $A$. Then there is a finite equivalence relation
$E(\bar{x},\bar{y})$ definable over $A$ such that $p_1(\bar{x})\cup p_2(\bar{y})\models \lnot E(\bar{x},\bar{y})$. 
\end{thm}


Shelah observed that ``seeing'' equivalence 
classes of definable equivalence relations as real elements gives a mild expansion of our 
theory, which we denote by $T^{eq}$, with many useful properties (we refer the reader to \cite[p.10]{PillayStability} for the 
construction). In this setting we denote by $acl^{eq}$ (respectively $dcl^{eq}$) the algebraic closure (respectively definable closure) 
calculated in $\mathbb{M}^{eq}$ (the monster model of $T^{eq}$, which actually is an expansion of $\mathbb{M}$).  

The following lemma is an easy application of the finite equivalence relation theorem.

\begin{lemma}\label{Stationarity} 
Let $A$ be a set of parameters in $\mathbb{M}$. Then $acl^{eq}(A)=dcl^{eq}(A)$ if and only if every type $p\in S_n(A)$ is stationary.
\end{lemma} 
\begin{proof}
($\Rightarrow$) Suppose, for the sake of contradiction, 
that there is a type $q$ in $S_n(A)$ which is not stationary. 
Let $q_1,q_2$ be two distinct non-forking extensions of $q$ to some set $B\supset A$. 
By the finite equivalence relation theorem, we may take $B$ to be $acl^{eq}(A)$. Now the hypothesis 
yields a contradiction as a type over $A$ extends uniquely to $dcl^{eq}(A)$.\\
($\Leftarrow$) Suppose that there is $e$ in $acl^{eq}(A)\setminus dcl^{eq}(A)$. 
Then $tp(e/A)$ is not stationary. Indeed, consider two distinct images of $e$ under automorphisms fixing $A$, then 
these images have different types over $acl^{eq}(A)$ and these types do not fork over $A$.
By the construction of $T^{eq}$, there is a tuple $\bar{b}$ in $\mathbb{M}$, such that $e\in dcl^{eq}(\bar{b})$. 
Now, it is easy to see that $tp(\bar{b}/A)$ is not stationary.
\end{proof}






\subsection{Stable groups}
A group, $\mathcal{G}:=(G,\cdot,\ldots)$, in the sense of model theory is a structure 
equipped with a group operation, but possibly also with some additional
relations and functions. In the case where all additional
structure is definable by multiplication alone, we speak of a {\em pure group}.

We define a {\em stable group} to be a group whose first-order theory $\mathcal{T}h(G,\cdot,\ldots)$ is stable. 
Although for the purpose of this paper it would be enough to consider pure groups 
there is no harm in developing stable group theory in greater generality. 
All results in this subsection can be found in \cite{PoizatStableGroups}.

\begin{defi}
Let $\mathcal{G}$ be a group. We say $\mathcal{G}$ is connected if there is no definable
proper subgroup of finite index.
\end{defi}

\begin{defi}
Let $\mathcal{G}$ be a stable group. Let $X$ be a definable subset of $G$.
Then $X$ is left (right) generic if finitely many left (right) translates of $X$ by
elements of $G$, cover $G$.
\end{defi}

As in a stable group $\mathcal{G}$ a definable set $X\subseteq G$ is left generic 
if and only if it is right generic, we simply say generic.

\begin{defi}
Let $\mathcal{G}$ be a stable group and let $A\subseteq G$. A type $p(x)\in S_1(A)$ is generic if every formula in $p(x)$ is generic.
\end{defi}


\begin{lemma}
Let $p(x)$ be a generic type of the stable group $\mathcal{G}$. Then any non-forking extension 
of $p(x)$ is generic.
\end{lemma}

It is not hard to see the following:

\begin{fact}
Let $\mathcal{G}$ be a stable group. Then $\mathcal{G}$ is connected if and only if there is, over any set of parameters, a unique generic type.
\end{fact}

This has an immediate corollary.

\begin{cor}\label{gensta}
Let $\mathcal{G}$ be a connected stable group. Then every generic type is stationary. In fact, generic types 
are exactly the non-forking extensions of the generic type over $\emptyset$.
\end{cor}
 
\subsection{Torsion-free hyperbolic groups}

In this subsection we see torsion-free hyperbolic groups as $\mathcal{L}$-structures in their 
natural language $\mathcal{L}:=\{ \cdot, ^{-1},1 \}$, i.e. the language of groups. We denote 
by $\F_n:=\langle e_1,\ldots e_n\rangle$ the free group of rank $n$ and we assume that $n>1$. We start 
with the deep result of Kharlampovich-Myasnikov \cite{KharlampovichMyasnikov} and Sela \cite{Sel6}.

\begin{thm}
Let $\F$ be a non abelian free factor of $\F_n$. Then $\F\prec\F_n$.
\end{thm}

This allows us to denote by $T_{fg}$ the common theory of non abelian free groups. Since 
connectedness (in the sense defined in the previous subsection) is a first-order property, we can state a result of 
Poizat \cite{PoizatGenericAndRegular} in the following way.

\begin{thm}
$T_{fg}$ is connected.
\end{thm}

As a matter of fact the theory of every (non-cyclic) torsion-free hyperbolic group is connected (see \cite{OuldHoucineHomogeneity}).

\begin{thm}
Let $G$ be a torsion-free hyperbolic group not elementarily equivalent to a free group. Then $\mathcal{T}h(G)$ is connected.
\end{thm}

On the other hand Sela \cite{SelaStability} proved the following: 

\begin{thm}
Let $G$ be a (non-cyclic) torsion-free hyperbolic group. Then $\mathcal{T}h(G)$ is stable.
\end{thm}

Thus, in every theory of a (non-cyclic) torsion-free hyperbolic group there is a unique generic type over any set of parameters. 

We specialize in $T_{fg}$ and we denote 
by $p_0$ the generic type over the empty set. By Corollary \ref{gensta} the generic type $p_0$ is stationary, thus we can
define $p_0^{(\kappa)}$ to be the type of $\kappa$-independent realizations of $p_0$. Pillay \cite{PillayGenericity} gave an 
understanding of the generic type in terms of its solution set in $\F_n$, in fact 
he proved more generally

\begin{thm} \label{GenericInFree}
An $m$-tuple $a_1,\ldots,a_m$ realizes $p_0^{(m)}$ in $\F_n$ if and only if $a_1,\ldots,a_m$ is part of a basis of $\F_n$. 
\end{thm}

In particular, the above theorem states that $tp^{\F_n}(e_1)$ is generic. 

An immediate consequence is that if $\F_n=\F*\F'*\F''$, then $\F \underset{\F'}{\forkindep} \F''$. 
We note that the above theorem has been generalized by the authors to finitely generated models of $T_{fg}$ with the appropriate 
modifications (see \cite{PerinSklinosHomogeneity}). 

The following theorem has been proved by the authors \cite{PerinSklinosHomogeneity} and Ould Houcine \cite{OuldHoucineHomogeneity} independently.

\begin{thm}
$\F_n$ is homogeneous.
\end{thm}

As a matter of fact we will see in Section \ref{IsolatingSec} that the proof of this result can be adapted to give the following 

\begin{thm}
Let $G$ be a torsion-free hyperbolic group concrete with respect to a subgroup $A$. Then $G$ is atomic over $A$.
\end{thm}

\section{Forking over free factors} \label{FreeFactorsSec}
In this section we describe forking independence in non abelian free groups over (possibly trivial) free factors. 
We begin with a result of more general interest.

\begin{thm}\label{stat}
Let $A\subset \F_n$. Then every type in $S_m(A)$ is stationary if and only if $tp^{\F_n}(e_1,\ldots,$ $e_n/A)$ is stationary.
\end{thm}
\begin{proof}
For the non trivial direction it is enough, by Lemma \ref{Stationarity}, to prove that $acl^{eq}(A)=dcl^{eq}(A)$. 
Let $a\in acl^{eq}(A)$. Then $a\in dcl^{eq}(e_1,\ldots,e_n)$, and since $tp^{\F_n}(e_1,\ldots,e_n/A)$ is stationary, 
we have that $tp^{\F_n^{eq}}(a/A)$ is stationary. Thus, $a\in dcl^{eq}(A)$ as we wanted. 
\end{proof}

We get the following corollaries.

\begin{cor}\label{emptyset}
Let $p(\bar{x})\in S_m(T_{fg})$. Then $p(\bar{x})$ is stationary.
\end{cor}
\begin{proof}
Since $p_0(x)$ is stationary, it follows that $p_0^{(2)}(x,y):=tp^{\F_2}(e_1,e_2)$ is. Now use 
Theorem \ref{stat} for $A=\emptyset$.  
\end{proof}

\begin{cor}\label{Z}
Suppose $a$ realizes $p_0$ in some model of $T_{fg}$ and $p\in S_m(a)$. 
Then $p$ is stationary.
\end{cor}
\begin{proof}
Suppose $b$ realizes the unique non-forking extension of $p_0$ over $a$. Then $\langle a,b\rangle\cong \F_2$, and $tp^{\langle a,b\rangle}(a,b/a)$ is 
stationary. Now use Theorem \ref{stat} for $A=\{a\}$.
\end{proof}

We are now ready to describe forking independence over free factors. 
For $m<n<\omega$, we will denote the free group of rank $n-m$ generated by $e_{m+1}\ldots,e_n$ by $\F_{m,n}$.

\begin{thm} \label{FreeFactors}
Let $\bar{a},\bar{b}\in \F_n$ and let $A$ be a free factor of $\F_n$. Then 
$\bar{a} \underset{A}{\forkindep} \bar{b}$ if and only if $\F_n$ admits a free decomposition $\F_n=\F*A*\F'$ with 
$\bar{a}\in \F*A$ and $\bar{b}\in A*\F'$.
\end{thm}

\begin{proof}
$(\Leftarrow)$ This direction is immediate as a basis of $\F_n$ is an independent set over $\emptyset$ by Theorem \ref{GenericInFree}. \\
$(\Rightarrow)$ We may assume that $A=\F_m$ for some $m<n$ (we also include the case where $\F_m$ is trivial). 
Let $\bar{a}(x_1,\ldots,x_n)$ be a tuple of words 
in variables $x_1,\ldots,x_n$, such that $\bar{a}(e_1,\ldots,e_n)=\bar{a}$. We consider the tuple 
$\bar{a}':=\bar{a}(e_1,\ldots,e_m,e_{n+1},\ldots,e_{2n-m})$ in $\F_{2n-m}$. As $e_{n+1},\ldots,$ $e_{2n-m}$ is independent from 
$e_1,\ldots,e_{n}$ over $e_1,\ldots,e_m$, we have that $\bar{a}'$ is independent from $\F_m\bar{b}$ over $\F_m$. 
We also note that $p:=tp^{\F_{2n-m}}(\bar{a}/\F_m)=tp^{\F_{2n-m}}(\bar{a}'/\F_m)$ as there is an automorphism of $\F_{2n-m}$ 
fixing $\F_m$ taking 
$\bar{a}$ to $\bar{a}'$. But $p$ is stationary (for $\F_m$ trivial follows from Corollary \ref{emptyset}, for $\F_m\cong\mathbb{Z}$ 
follows from Corollary \ref{Z}, and in any other case $\F_m$ is a model so this follows from Fact \ref{forkprop}(v)), 
thus $tp^{\F_{2n-m}}(\bar{a}/\F_m\bar{b})=tp^{\F_{2n-m}}(\bar{a}'/\F_m\bar{b})$. By homogeneity of $\F_{2n-m}$ there is 
an automorphism $f\in Aut(\F_{2n-m}/\bar{b})$ which sends $\bar{a}'$ to $\bar{a}$. We consider the following decomposition 
$\F_{2n-m}=\F_m*\F_{m+1,n}*\F_{n+1,2n-m}$. We now apply $f$ to this decomposition and we get 
$\F_{2n-m}=\F_m*f(\F_{m+1,n})*f(\F_{n+1,2n-m})$ with $\bar{b}\in \F_m*f(\F_{m+1,n})$ and $\bar{a}\in \F_m*f(\F_{n+1,2n-m})$.\\
But $\F_n$ is a subgroup of $\F_{2n-m}$, thus by Kurosh subgroup theorem we get a decomposition of $\F_n$ as 
we wanted.  
\end{proof}

\section{JSJ decompositions and modular groups} \label{JSJSec}
The main theme of this section is the description of the modular group $\Mod_A(G)$ of 
automorphisms of a torsion-free hyperbolic group $G$ which is freely indecomposable with respect to a subgroup $A$. We briefly explain the outline of the section and the tools used. 

In the first subsection we are concerned with cyclic JSJ splittings of $G$ relative to $A$. These are splittings of a group $G$ which in some sense encode all the splittings of $G$ over cyclic subgroups in which $A$ is elliptic. 

The next three subsections are devoted to ``elementary'' automorphisms associated to a splitting of a group. These are automorphisms that 
can be read locally from the splitting, namely Dehn twists and vertex automorphisms. 
Under certain conditions we prove that ``elementary'' 
automorphisms almost commute. 

In the final subsection we prove that one can read the modular group of $G$
from its JSJ splitting relative to $A$, and we moreover give a normal form theorem for the modular automorphisms. These results are not new (see \cite{LevittAMOfHypGroups} and \cite{GuirardelLevittSplittingsAndAM}), but the hands-on proofs we give in this specific setting hopefully helps to gain low-level intuition. 
We also describe the normal form of automorphisms which fix pointwise a subgroup of $G$ containing $A$.

\subsection{G-trees and JSJ decompositions}

We will use definitions and results about graphs of groups from \cite{SerreTrees}. 
Let $\Gamma$ be a graph: we will denote by $V(\Gamma)$ and $E(\Gamma)$ respectively its vertex and edge sets. 
The set $E(\Gamma)$ is always assumed to be stable under the involution which to an edge $e$ associates its inverse edge $\bar{e}$. 

Let $G$ be a finitely generated group. A {\em$G$-tree} is a simplicial tree $T$ 
endowed with an action of $G$ without inversions of edges. 
We say $T$ is {\em minimal} if it admits no proper $G$-invariant subtree.  
A {\em cyclic $G$-tree} is a $G$-tree whose edge stabilizers are infinite cyclic. 
If $A$ is a subset of $G$, a {\em $(G,A)$-tree} is a $G$-tree in which $A$ fixes a point. 
Following \cite{GuirardelLevittDefSpaces}, we call a (not necessarily simplicial) surjective equivariant 
map $d: T_1 \to T_2$ between two $(G,A)$-trees a {\em domination map}. 
A surjective simplicial map $p:T_1 \to T_2$ which consists in collapsing some orbits of edges to points 
is called a {\em collapse map}. In this case, we also say that $T_1$ {\em refines} $T_2$.

We also define:

\begin{defi} \emph{(Bass-Serre presentation)} Let $G$ be a finitely generated group, and let $T$ be a $G$-tree. Denote by $\Lambda$ the 
corresponding quotient graph of groups and by $p$ the quotient map $T \to \Lambda$. 

A Bass-Serre presentation for $\Lambda$ is a triple $(T^1, T^0, (t_e)_{e  \in E_1)})$ consisting of
\begin{itemize}
\item a subtree $T^1$ of $T$ which contains exactly one edge of $p^{-1}(e)$ for each edge $e$ of $\Lambda$;
\item a subtree $T^0$ of  $T^1$ which contains exactly one vertex of $p^{-1}(v)$ for each vertex $v$ of $\Lambda$;
\item for each edge $e \in E_1:= \{ e=uv \mid u \in T^0, v \in T^1\setminus T^0 \}$, an element $t_e$ of $G$ such that $t_e^{-1} \cdot v$ 
lies in $T^0$.
\end{itemize}
We call $t_e$ the stable letter associated to $e$. 
\end{defi}
One can give an explicit presentation of the group $G$ whose generating set is the union of the stabilizers of vertices of 
$T^0$ together with the stable letters $t_e$, hence the name.

For JSJ decompositions, we will use the framework described in \cite{GuirardelLevittJSJI} and \cite{GuirardelLevittJSJII} 
(see also the brief summary given in Section 3 of \cite{PerinSklinosHomogeneity}). 
We recall here the main definitions and results we will use. Unless mentioned otherwise, all $G$-trees are assumed to be minimal.


\paragraph{Deformation space.} The {\em deformation space} of a cyclic $(G,A)$-tree $T$ is the set of all cyclic $(G,A)$-trees $T'$ such that $T$ dominates 
$T'$ and $T'$ dominates $T$. A cyclic $(G,A)$-tree is {\em universally elliptic} if its edge stabilizers are elliptic 
in every cyclic $(G,A)$-tree. If $T$ is a universally elliptic cyclic $(G,A)$-tree, and $T'$ is any cyclic $(G,A)$-tree, 
it is easy to see that there is a tree $\hat{T}$ which refines $T$ and dominates $T'$ (see \cite[Lemma 3.2]{GuirardelLevittJSJI}).

\paragraph{JSJ trees.} A cyclic relative {\em JSJ tree} for $G$ with respect to $A$ is a universally elliptic cyclic $(G,A)$-tree which dominates 
any other universally elliptic cyclic $(G,A)$-tree. All these JSJ trees belong to a same deformation space, 
that we denote ${\cal D}_{JSJ}$. Guirardel and Levitt show that if $G$ is finitely presented and $A$ is finitely generated, 
the JSJ deformation space always exists (see \cite[Theorem 5.1]{GuirardelLevittJSJI}). It is easily seen to be unique.

\paragraph{Rigid and flexible vertices.} A vertex stabilizer in a (relative) JSJ tree is said to be {\em rigid} 
if it is elliptic in any cyclic $(G,A)$-tree , and {\em flexible} if not. In the case of a torsion-free hyperbolic group $G$ and a finitely 
generated subgroup $A$ of $G$ with respect to which $G$ is freely indecomposable, 
the flexible vertices of a cyclic JSJ tree of $G$ with respect to $A$ are 
{\em surface type} vertices \cite[Theorem 8.20]{GuirardelLevittJSJI}, i.e. their stabilizers are fundamental groups of hyperbolic surfaces with boundary, 
any adjacent edge group is contained in a maximal boundary subgroup, and any maximal boudary 
subgroup contains either exactly one adjacent edge group, or exactly one conjugate of $A$ \cite[Remark 8.19]{GuirardelLevittJSJI}. 
Note that (since the vertices are not rigid) these surfaces cannot be thrice punctured spheres \cite[Remark 8.19]{GuirardelLevittJSJI}. Also, they cannot be once punctured Klein bottle or twice punctured projective planes. Indeed, otherwise the JSJ tree $T$ can be refined to a tree $\hat{T}$ by the splitting of this surface corresponding to one or two curves bounding M\"obius bands. This new $(G,A)$-tree is still universally elliptic, since there are no incompatible splittings of the surface, but $T$ does not dominate $\hat{T}$: we get a contradiction.

We give a simple example of a JSJ decomposition at the level of graph of groups. 

\begin{figure*}[ht!]\label{Fig1}

\centering
\includegraphics[width=.5\textwidth]{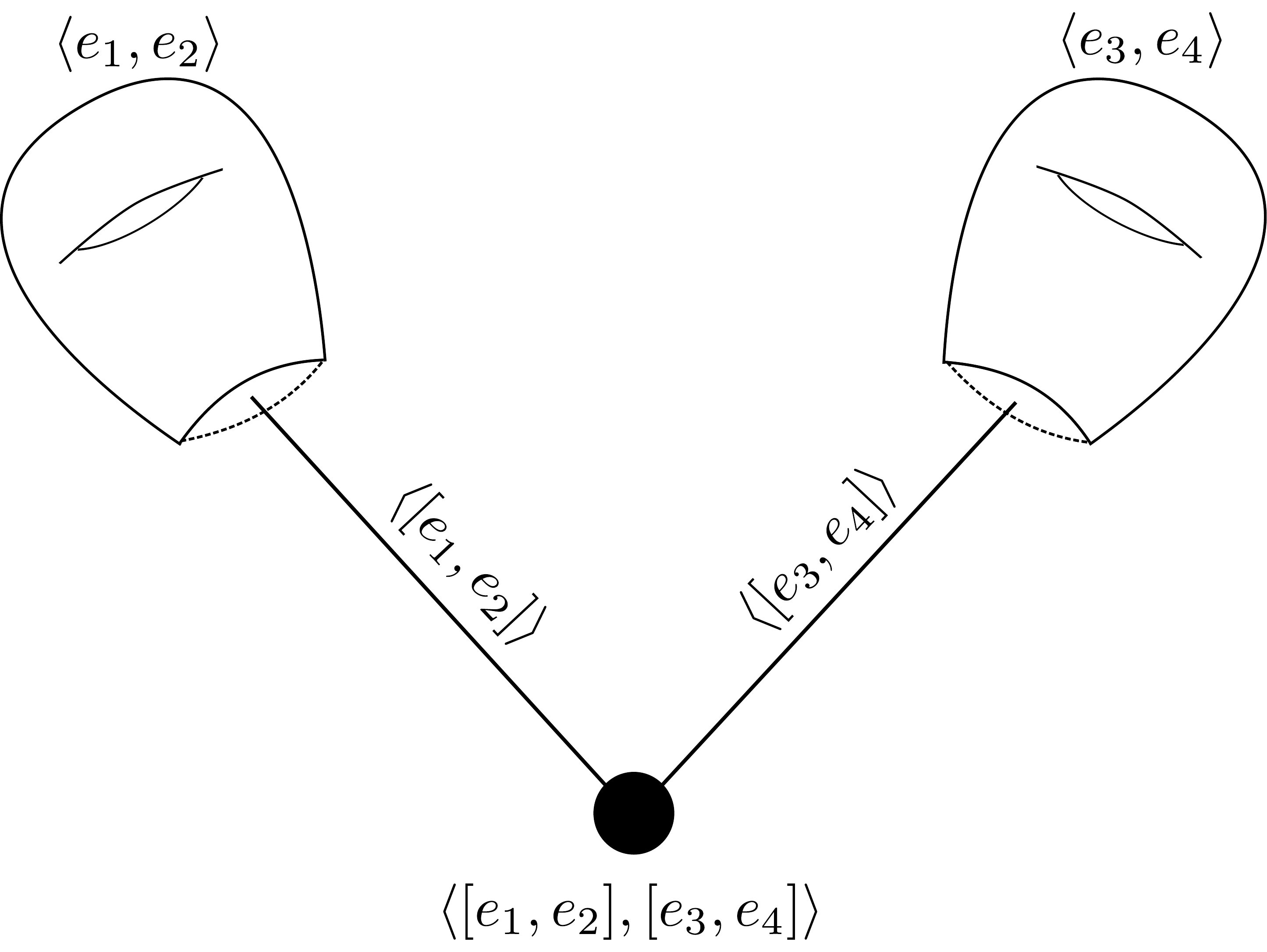}
\caption{A JSJ decomposition of the free group $\F_4$ on $e_1, \ldots, e_4$ relative to $A = \langle [e_1,e_2], [e_3,e_4] \rangle$.}

\end{figure*}


%

\paragraph{The tree of cylinders.} In \cite{GuirardelLevittTreeOfCylinders}, {\em cylinders} in cyclic $G$-trees are defined as equivalence 
classes of edges under the equivalence relation given by commensurability of stabilizers, 
and to any $G$-tree $T$ is associated its {\em tree of cylinders}. 
It can be obtained from $T$ as follows: the vertex set is the union 
$V_0(T_c) \cup V_1(T_c)$ where $V_0(T_c)$ contains a 
vertex $w'$ for each vertex $w$ of $T$ contained in at least two distinct cylinders, and $V_1(T_c)$ contains a vertex 
$v_c$ for each cylinder $c$ of $T$. There is an edge between vertices $w'$ and $v_c$ lying in $V_0(T_c)$ and $V_1(T_c)$ respectively 
if and only if $w$ belongs to the cylinder $c$.  

We get a tree which is bipartite: every edge in the tree of cylinders joins a vertex from $V_0(T_c)$ (which is cyclically stabilized) 
to a vertex of $V_1(T_c)$. Since the action of $G$ on $T$ sends cylinders to cylinders, the tree of cylinder admits an obvious $G$ action. 
Note also that if $H$ stabilizes an edge $e$ of $T$, its centralizer $C(H)$ preserves the cylinder containing $e$ since 
the translates of $e$ are also stabilized by $H$: in particular there is a vertex in $T_c$ whose stabilizer is $C(H)$. 
It is moreover easy to see that this vertex is unique.

It turns out that the tree of cylinders is in fact an invariant of the deformation space 
\cite[Corollary 4.10]{GuirardelLevittTreeOfCylinders}.

\paragraph{Case of freely indecomposable torsion-free hyperbolic groups.} By \cite[Theorem 2]{GuirardelLevittTreeOfCylinders}, 
if $G$ is a torsion-free hyperbolic group freely indecomposable with respect to a finitely generated subgroup $A$, 
the tree of cylinders $T_c$ of the cyclic JSJ deformation space of 
$G$ with respect to $A$ is itself a JSJ tree, and it is moreover strongly $2$-acylindrical: namely, 
if a non-trivial element stabilizes two distinct edges, they are adjacent to a common cyclically stabilized vertex. 

Moreover, in this case the tree of cylinders is not only universally elliptic, but in fact universally compatible: namely, 
given any cyclic $(G,A)$-tree $T$, there is a refinement $\hat{T}$ of $T_c$ which \emph{collapses} onto 
$T$ \cite[Theorem 6]{GuirardelLevittTreeOfCylinders}.

The JSJ 
deformation space being unique, it must be preserved under the action of $\Aut_A(G)$ on (isomorphism classes of) $(G,A)$-trees defined by twisting of the $G$-actions. 
Thus the tree of cylinder is a fixed point of this action, that is, 
for any automorphism $\phi \in \Aut_A(G)$, there is an automorphism 
$f:T_c \to T_c$ such that for any $x \in T_c$ and $g \in G$ we have $f(g \cdot x) = \phi(g) \cdot f(x)$.

\paragraph{JSJ relative to a non finitely generated subgroup.} 
Let $G$ be a torsion-free hyperbolic group freely indecomposable with respect to a subgroup $A$. 
By \cite[Proposition 3.7]{PerinSklinosHomogeneity}, there is a finitely generated subgroup 
$A_0$ of $A$ such that $G$ is freely indecomposable with respect to $A_0$ and $A$ 
is elliptic in any cyclic JSJ tree of $G$ with respect to $A_0$. 
The tree of cylinder of the cyclic JSJ deformation space with respect to $A_0$ clearly admits a 
common refinement with any cyclic $(G,A)$-tree, and satisfies all the 
properties we described above in the case $A$ was finitely generated. 
So whenever we refer to the tree of cylinders of the cyclic JSJ deformation space with respect to $A$ (for a possibly 
non finitely generated group), 
we tacitly mean the tree of cylinders of the cyclic JSJ deformation space with respect to $A_0$.  

\paragraph{The pointed cyclic JSJ tree.}
For our purposes, we need a tree with a basepoint which is a slight variation of the tree of cylinders. Note that this tree is not minimal. 
\begin{defi} Let $G$ be a torsion-free hyperbolic group freely indecomposable with respect to a subgroup $A$. 
Let $T_c$ be the tree of cylinders of the cyclic JSJ deformation space of $G$ with respect to $A$. 

We define the pointed cyclic JSJ tree $(T, v)$ of $G$ with respect to $A$ as follows 
\begin{itemize}
 \item if $A$ is cyclic, let $u$ be either, if it exists, the unique vertex whose stabilizer is exactly the centralizer 
  $C(A)$ of $A$, or otherwise, the unique vertex fixed by $C(A)$. We take $T$ to be the tree $T_c$ to which we add one 
  orbit of vertices $G.v$, one orbit of edges $G.e$ with $e = vu$, and we set $\Stab(v) = \Stab(e) = C(A)$.
 \item if $A$ is not cyclic, we take $T=T_c$ and we let $v$ be the unique vertex fixed by $A$.
\end{itemize}
\end{defi}

\begin{defi} A vertex of the pointed cyclic JSJ tree is said to be a Z-type vertex if it is cyclically 
stabilized and distinct from the basepoint $v$.

\end{defi}

\begin{rmk} \label{AcylindricalAndUnivCompatible} It is not hard to see that the pointed cyclic JSJ tree of $G$ with respect to 
$A$ is strongly $2$-acylindrical, universally compatible, and a fixed point of the action of $\Aut_A(G)$ on $(G,A)$-trees defined by twisting of the $G$-action.
\end{rmk}

\subsection{Dehn twists}

Let $G$ be a finitely generated group. 

\begin{defi} Let $e=uv$ be an edge in a $G$-tree $T$, and let $a$ be an element in the centralizer in $G$ of $\Stab(e)$. 
The $G$-tree $T'$ obtained from $T$ by collapsing all the edges not in the orbit of 
$e$ induces a splitting of $G$ as an amalgamated product $G = U*_{\Stab(e)} V$ or as an 
HNN extension $U*_{\Stab(e)}$ with stable letter $t$, where $U$ is the stabilizer of the image vertex of $u$ in $T'$. 

The Dehn twist by $a$ about $e$ is the automorphism of $G$ which restricts to the identity on 
$U$ and to conjugation by $a$ on $V$ (respectively sends $t$ to $at$ in the HNN case).
\end{defi}

The proof of the following lemma is immediate. We first recall that a $G$-tree is called 
{\em non-trivial} if there is no globally fixed point. It is 
not hard to see that if $G$ is finitely generated and $T$ is a non-trivial $G$-tree then
$T$ contains a unique minimal $G$-invariant subtree.
 
\begin{lemma} \label{DisjointDehnTwistFixes} Let $G$ be a finitely generated group, and let $T$ be a cyclic $G$-tree. 
Suppose $H$ is a finitely generated subgroup of $G$, whose minimal subtree $T_{H}$ in $T$ contains 
no translate of $e$. Then any Dehn twist about $e$ restricts to a conjugation on $H$ by an 
element which depends only on the connected component of $T\setminus G.e$ containing $T_H$.
\end{lemma}

The following lemma describes Dehn twists with respect to Bass-Serre presentations.

\begin{lemma} \label{DehnTwistOnBassSerre} Let $G$ be a finitely generated group, and let $T$ be a 
cyclic $G$-tree with a Bass-Serre presentation $(T^1, T^0, (t_f)_{f\in E_1})$. Let $\tau_e$ be a Dehn 
twist by an element $a$ about an edge $e=uv$ of $T^1$. Then
\begin{itemize}
 \item for each vertex $x$ of $T^0$, the restriction of $\tau_e$ to $G_x$ is a conjugation by an element 
 $g_x$ which is $1$ if $x$ and $u$ are in the same connected component of $T^1\setminus\{e\}$, and $a$ otherwise;
 \item for any edge $f = xy'$ of $T^1$ with $x \in T^0$ and $t^{-1}_f \cdot y' = y \in T^0$, we have 
$$\tau_e(t_f) =
\begin{cases}
g_x t_f g^{-1}_{y} & \textrm{ if } f \neq e \\
a t_f & \textrm{ if } f=e \\ 
t_f a^{-1} & \textrm{ if } f=\bar{e}
\end{cases}
$$
\end{itemize}
\end{lemma}

\begin{figure}[ht!]

\centering
\includegraphics[width=.99\textwidth]{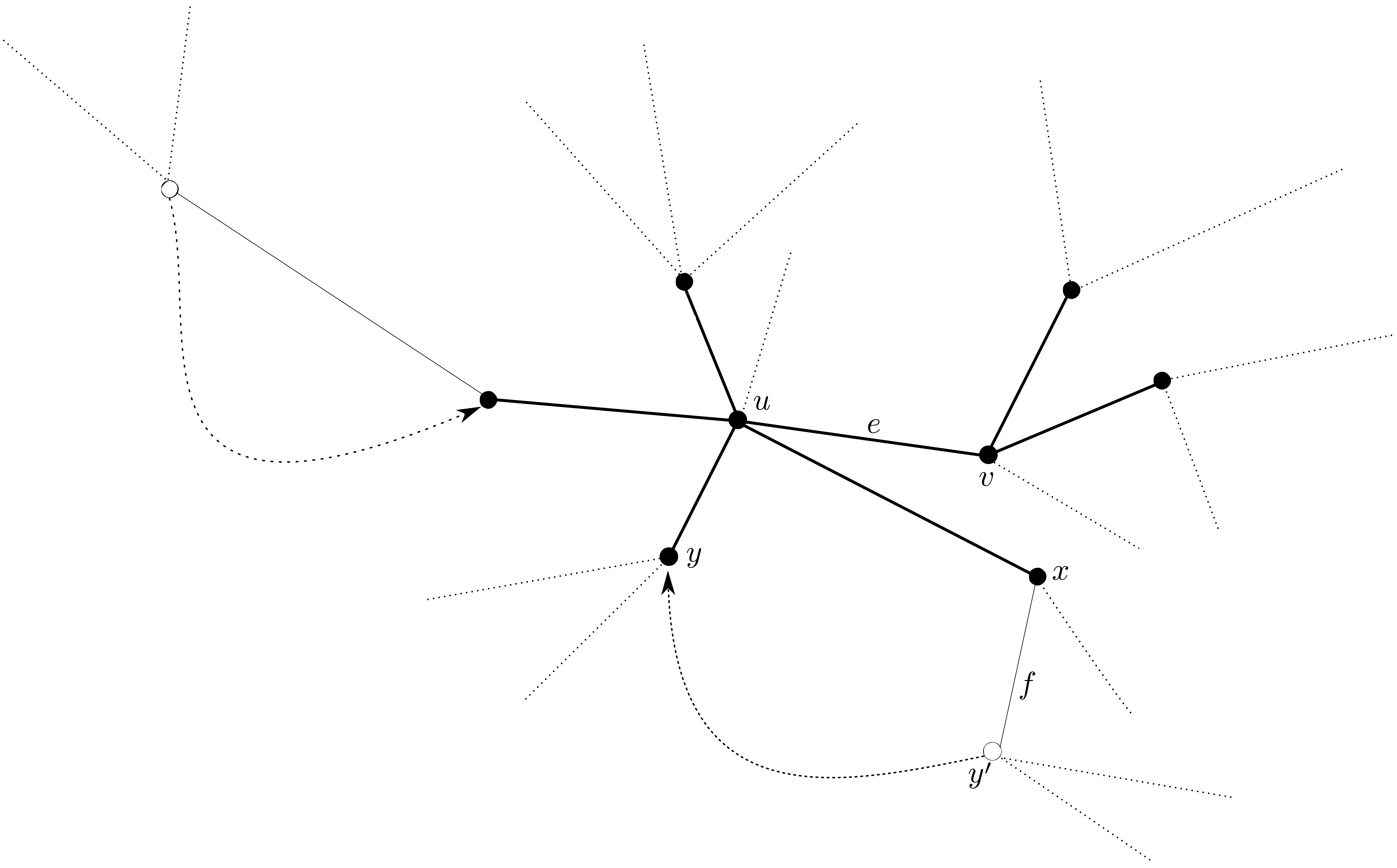}
\caption{A choice of $T^1$ and $T^0$ (thick subtree).}

\end{figure}

\begin{proof} The proof is straightforward, it suffices to note that the images of $x,y$ and $y'$ 
under the map $p$ which collapses all the edges not in the orbit of $e$ all belong to $\{ p(u), p(v)\}$, and to consider the various possibilities.
\end{proof}

\newpage

\begin{rmk} \label{DehnTwistFixingAGivenVertex} If $\tau$ is the Dehn twist by $a$ about $e$, and $\tau'$ 
the Dehn twist by $a^{-1}$ about $\bar{e}$, we have $\tau = \Conj(a) \circ \tau'$.

In particular, if $(T^1, T^0, (t_f)_f)$ is a Bass-Serre presentation for $T$ such that $e$ is in $T^1$, and if $R$ 
is a connected component of $T^0\setminus \{e\}$, then there exists an element $g$ such that $\Conj(g) \circ \tau$ 
is a Dehn twist about $e$ or $\bar{e}$ which restricts to the identity on $G_x$ for any vertex $x$ of $R$.
\end{rmk}

The next lemma gives a useful relation between Dehn twists about edges adjacent to a common cyclically stabilized vertex:
\begin{lemma} \label{RelationDehnTwists} Let $G$ be a finitely generated group, and let $T$ be a $G$-tree. 
Suppose $v$ is a vertex of $T$ whose stabilizer is cyclic, and let $e_1 = u_1v, \ldots, e_r=u_rv$ be 
representatives of the orbits of edges adjacent to $v$. Let $z$ be an element in the centralizer of $\Stab(v)$, and 
denote by $\tau_i$ the Dehn twist about $e_i$ by $z$. Then we have:
$$ \tau_{1} \ldots \tau_{r} = \Conj(z^{r-1}).$$
\end{lemma}

\begin{proof}Choose a Bass-Serre presentation $(T^1, T^0, (t_f)_{f})$ such that $v \in T^0$ and 
all the edges $e_i$ are contained in $T^1$. 

It is easy to see that both $\tau_1 \ldots \tau_r$ and $\Conj(z^{r-1})$ restrict to the identity on $\Stab(v)$ and on $\langle z \rangle$. 
If $w$ is a vertex of $T^0$ distinct from $v$, the Dehn twist $\tau_{e_i}$ restricts on $\Stab(w)$ to a conjugation by an element 
$g^i_w$ which is $1$ if $w$ lies in the same connected component of $T^1\setminus\{e_i\}$ as $u_i$, and $z$ otherwise. 
Now the first alternative holds for exactly one value of $i$, thus $\tau_{e_1} \ldots \tau_{e_r}$ restricts to a conjugation by $z^{r-1}$ on $\Stab(w)$. 

If $f=wx'$ is an edge of $T^1\setminus T^0$ with $w$ in $T^0$, note first that $f$ is distinct from all the edges $e_i$ (though we can have $f=\bar{e}_i$). 
By Lemma \ref{DehnTwistOnBassSerre}, if $x = t^{-1}_f \cdot x'$ we have 
$$\tau_i(t_f) = \begin{cases}
g^i_w t_f (g^i_x)^{-1} & \textrm{ if } f \neq \bar{e}_i \\
t_f z^{-1} & \textrm{ if } f = \bar{e}_i \\
\end{cases}$$

If $f$ is distinct from $\bar{e}_i$ for all values of $i$ we conclude as 
before by noting that $g^i_w$ (respectively $g^i_x$) is $z$ for all but one value of $i$. If $f= \bar{e}_i$, 
then $w=v$, $x'=u_i$ and $x$ is not in the same connected component as $u_i$, so $g^j_w = z$ for all $j$, and $g^j_x=z$ 
for all but one value of $j$, and this value cannot be $i$. Thus in both cases, we get that 
$\tau_{1} \ldots \tau_{r} (t_f) = z^{r-1} t_f z^{1-r}$.
\end{proof}

\subsection{Vertex automorphisms}

We want to extend automorphisms of stabilizers of vertices in a $G$-tree to automorphisms of $G$. For this we give

\begin{defi} Let $G$ be a finitely generated group acting on a tree $T$, and let $v$ be a vertex in $T$. 
Denote by $p$ the map collapsing all the orbits of edges of $T$ except those of the edges adjacent to $v$.

An automorphism $\sigma$ of $G$ is called a vertex automorphism associated to $v$ if $\sigma(G_v) = G_v$, 
and if for every edge $e = vw$ of $p(T)$ adjacent to $v$, it restricts to a conjugation by an element $g_e$ on the stabilizer of $e$, 
as well as on the stabilizer of $w$ if $w$ is not in the orbit of $v$.
\end{defi}

\begin{rmk} \label{BuildingVertexAM} If $v$ is a vertex in a $G$-tree $T$, and if $\sigma_0$ is
an automorphism of $\Stab_G(v)$ which restricts to a conjugation by an element $g_e$ on the stabilizer of each edge $e$ adjacent to $v$, 
we can extend $\sigma_0$ to a vertex automorphism $\sigma$ of $G$. 
For this, choose a Bass-Serre presentation $(T^1, T^0, (t_f)_f)$ for $G$ with respect to $p(T)$ such that $p(v) \in T^0$, 
and such that the orbits of edges adjacent to $p(v)$ are represented in $T^1$ by edges adjacent to $p(v)$. 
We then define $\sigma$ as follows: 
\begin{itemize}
 \item on $G_{p(v)}$, $\sigma$ restricts to $\sigma_0$;
 \item for any vertex $x$ of $T^0$ distinct from  $p(v)$, $\sigma$ restricts on $G_x$ to conjugation by $g_e$ where $e = p(v)x$;
 \item for $f=p(v)x'$ an edge of $T^1\setminus T^0$ with $x= t_f^{-1} \cdot x'$ in $T^0$, we set 
 $$\sigma(t_f) = \begin{cases}
g_ft_fg_e^{-1} \ \textrm{where} \ e=p(v)x & \textrm{ if } x \neq p(v) \\
g_ft_fg_{f'}^{-1} \ \textrm{where} \ f'=t_f^{-1}f & \textrm{ if } x = p(v) \\
\end{cases}$$ 
\end{itemize}
\end{rmk}

\begin{figure}[ht!]

\centering
\includegraphics[width=0.99\textwidth]{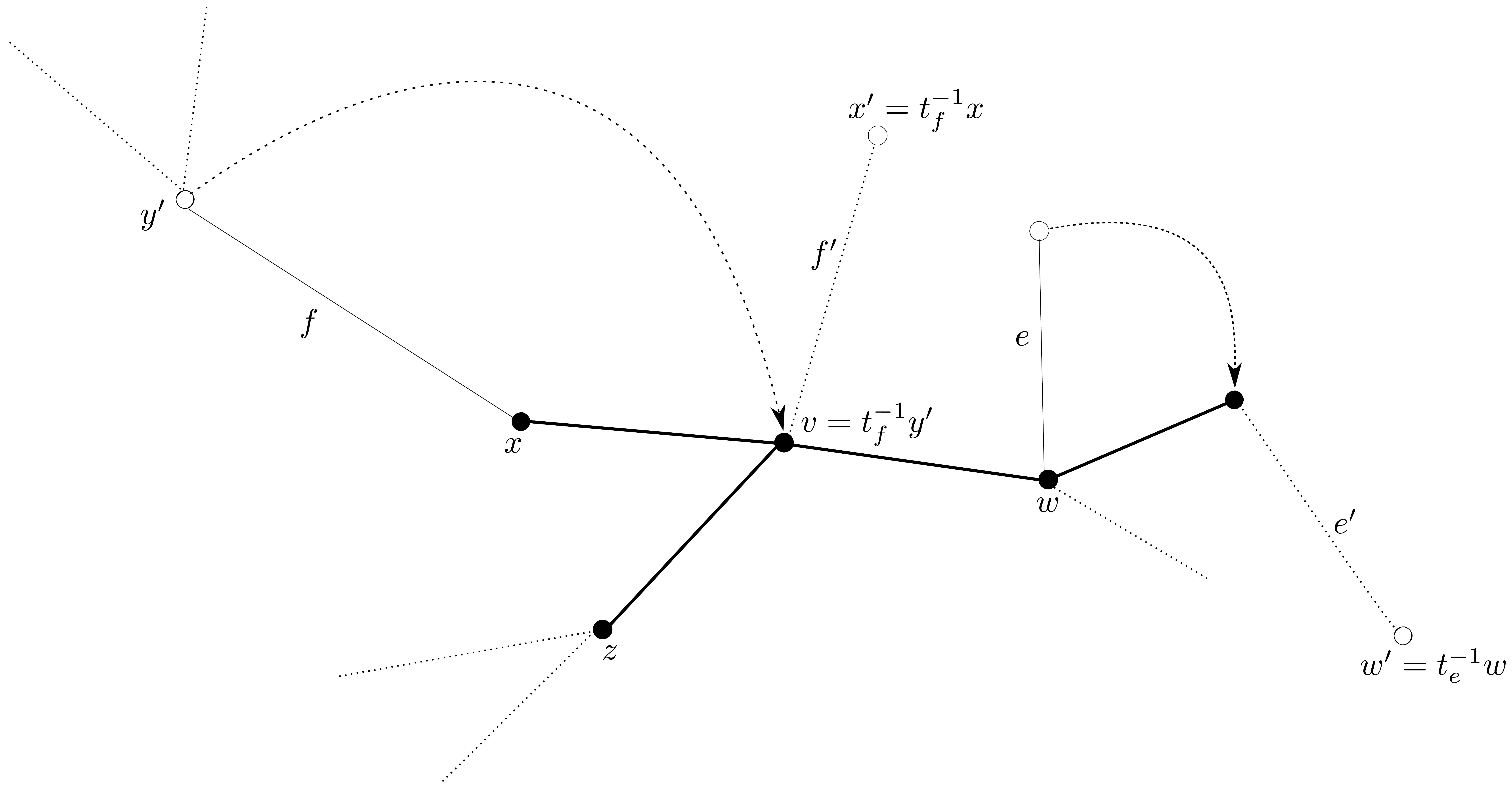}
\caption{A Bass-Serre presentation $(T^1,T^0,(t_f)_f)$ for the action of $G$ on $T$ 
together with the translates of the edges in $T^1\setminus T^0$ by the stable letters.}

\end{figure}

\begin{figure}[ht!]

\centering
\includegraphics[width=0.7\textwidth]{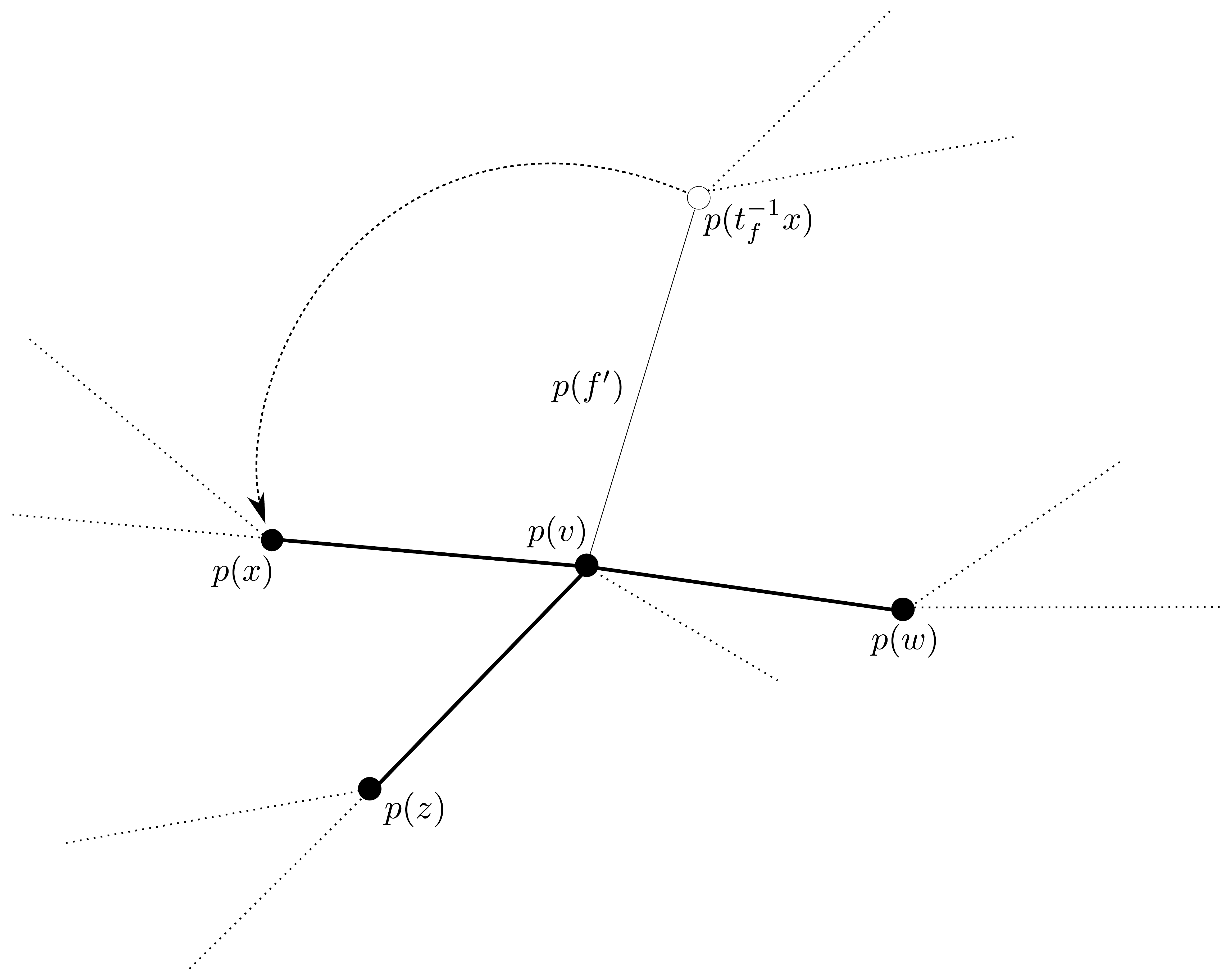}
\caption{A Bass-Serre presentation for the corresponding action of $G$ on $p(T)$.}

\end{figure}

\begin{rmk} \label{VertexAMDifferByDT} 
If $\sigma'$ is another vertex automorphism associated 
to $v$ such that $\sigma\mid_{G_v} = \sigma'\mid_{G_v}$, then for any vertex $w$ of $p(T)$ 
adjacent to $p(v)$ by an edge $e$, the restriction of $\sigma$ to $G_w$ is conjugation by an element $g'_{w}$ such that 
$g^{-1}_w g'_w$ lies the centralizer of $G_e$.

It is therefore easy to deduce that $\sigma^{-1} \circ \sigma'$ is a product of Dehn twists about edges of $p(T)$.

\end{rmk}

We now give an analogue of Lemma \ref{DisjointDehnTwistFixes} for vertex automorphisms.

\begin{lemma} \label{DisjointVertexAMFixes} Let $H$ be a finitely generated subgroup of $G$, and denote by $T_{H}$ the minimal subtree of $H$ in $T$. 

If no translate of $T_{H}$ contains $v$, any vertex automorphism $\sigma_v$ associated to $v$ restricts to a conjugation on $H$.
\end{lemma}

\begin{proof} The image of $T_H$ by $p$ is a vertex $x$ of $p(T)$, thus $\sigma_v$ restricts to a conjugation on $\Stab(x)$ which contains $H$. 
\end{proof}

The following lemma describes vertex automorphisms with respect to Bass-Serre presentations.
\begin{lemma} \label{VertexAMOnBassSerre} Let $G$ be a finitely generated group, and let 
$T$ be a $G$-tree. Let $(T^1, T^0, (t_f)_{f \in E_1})$ be a Bass-Serre presentation for $G$ with respect to $T$. 
Let $\sigma$ be a vertex automorphism of $G$ associated to a vertex $v$ of $T^0$. For a vertex $u$ or an edge 
$e$ of $T \setminus G.v$, we denote by $[u]$ (respectively $[e]$) 
the connected component of $T \setminus G.v$ containing $u$ (respectively $e$).

There exists an element $g_R$ of $G_v$ associated to each connected component $R$ of $T \setminus G.v$ adjacent to $v$, so that
\begin{itemize}
 \item for each vertex $u$ of $T^0 \setminus \{v\}$, the restriction of $\sigma$ to $G_u$ is conjugation by $g_{[u]}$;
 \item for any edge $f$ of $T^1\setminus T^0$ with $f' =t^{-1}_f \cdot f$, we have 
 $\sigma(t_f)= g_{[f]} z t_f (g_{[f']})^{-1}$ for $z$ in $C(\Stab (f))$. 
\end{itemize}
\end{lemma}

\begin{proof} Let $R$ be a connected component of $T\setminus G \cdot v$ adjacent to $v$. 
All the vertices in $R$ map by the collapse map $p$ to a same vertex $x_R$ of $p(T)$ 
adjacent to $p(v)$ and distinct from $p(v)$, so their stabilizers are contained in $\Stab(x_R)$, 
on which $\sigma$ restricts to conjugation by an element $g_R$ of $G_v$ by definition. 
If $R$ does not contain any vertices, it is reduced to a single edge $e$ adjacent to $v$ and 
we let $g_R$ be such that $\sigma$ restricts to a conjugation by $g_R$ on the stabilizer of $e$.

For the second point, note that $\sigma$ restricts to conjugation by $g_{[f]}$ on $\Stab(f)$, while it restricts to conjugation by 
$g_{[f']}$ on $\Stab(f') = t^{-1}_f \Stab(f) t_f$. Hence we get that for any element $h \in \Stab(f)$ we have 
$$ g_{[f']} t_f^{-1} \; h \; t_f g^{-1}_{[f']} = \sigma (t_f^{-1}h t_f) = \sigma(t_f)^{-1} g_{[f]}  \; h \;  g^{-1}_{[f']} \sigma(t_f) $$

This implies that $\sigma(t_f)= g_{[f]} z t_f g_{[f']}^{-1}$ for some element $z$ in the centralizer of $\Stab(f)$. 
\end{proof}

\begin{rmk} \label{VertexAMFixingAGivenVertex} 
Let $\sigma$ be a vertex automorphism with support $v \in T$, 
let $R$ be a connected component of $T \setminus G.v$ adjacent to $v$: for any vertex $x$ of $R$, $\sigma$ restricts on $G_x$ to 
conjugation by an element $g_R$ of $\Stab(v)$. 
Then $\Conj(g^{-1}_R) \circ \sigma$ is a vertex automorphism with support $v$, 
and it restricts to the identity on $G_x$ for any vertex $x$ of $R$. 
\end{rmk}

\subsection{Elementary automorphisms}

\begin{defi} Let $T$ be a $G$-tree. 
If $\rho$ is a Dehn twist about an edge $e$ of $T$, or a vertex automorphism associated to a vertex $v$ of $T$, 
we say it is an elementary automorphism associated to $T$. We call the edge $e$ 
(respectively the vertex $v$) the support of $\rho$ and denote it $\Supp(\rho)$. 
\end{defi}

\begin{lemma} \label{ElementaryAMOnTranslatedSupp} 
Suppose $\rho$ is an elementary automorphism associated to a $G$-tree $T$. 
Then for any $g \in G$, $Conj(g\rho(g^{-1})) \circ \rho$ is an elementary automorphism of support $g \cdot \Supp(\rho)$.
\end{lemma}

\begin{proof} 
Denote by $\rho'$ the automorphism $\Conj(g) \circ \rho \circ \Conj(g^{-1}) = Conj(g\rho(g^{-1})) \circ \rho$. 
It is easy to check (for example using Lemmas \ref{DehnTwistOnBassSerre} and \ref{VertexAMOnBassSerre}) 
that if $\rho$ is the Dehn twist about an edge $e$ by $z \in C(\Stab(e))$, then $\rho'$ is the Dehn twist about 
$g \cdot e$ by $gzg^{-1}$, and that if $\rho$ is a vertex automorphism associated to $v$, 
then $\rho'$ is a vertex automorphism associated to $g \cdot v$ which restricts to 
$\Conj(g) \circ \rho|_{G_v} \circ \Conj(g^{-1})$ on $G_{g \cdot v}$.
\end{proof}

Under some conditions on $G$ and $T$, elementary automorphisms commute up to conjugation.
\begin{prop} \label{ElementaryAMCommute} 
Let $T$ be a cyclic $G$-tree whose edge stabilizers have cyclic centralizers.

Let $\rho$ and $\sigma$ be two elementary automorphisms associated to $T$ which have supports lying in distinct orbits. 
Then there exists an element $g$ of $G$ such that
$$ \rho \circ \sigma = \Conj(g) \circ \sigma \circ \rho$$
\end{prop}

\begin{proof} Choose a Bass-Serre presentation $(T^1, T^0, (t_e)_e)$ for $G$ with respect to $T$. 
By Lemma \ref{ElementaryAMOnTranslatedSupp}, we may assume that both $\rho$ and $\sigma$ have support in $T^1$.

Let $R^0, \ldots , R^m$ denote the connected components of $T^1 \setminus \{\Supp(\rho)\}$, and $S^0, \ldots, S^n$ the 
connected components of $T^1 \setminus \{\Supp(\sigma)\}$, and assume without loss of generality that $\Supp(\sigma)$ 
lies in $R^0$ and $\Supp(\rho)$ lies in $S^0$. Note that $S^1, \ldots, S^n$ are contained in $R^0$, and that $R^1, \ldots, R^m$ are contained in $S^0$.




By Lemmas \ref{DehnTwistOnBassSerre} and \ref{VertexAMOnBassSerre}, there are elements $g_j$ (respectively $h_k$) 
such that $\rho$ (respectively $\sigma$) restricts to conjugation by $g_j$ (respectively $h_k$) on the 
stabilizer $G_u$ of each vertex $u$ which lies in $R^j$ (respectively $S^k$) and is not in the orbit of the support of $\rho$ (respectively $\sigma$). Also, if $f$ is an edge of $T^1$ which lies in $R^j$ (respectively $S^k$), then $\rho$ (respectively $\sigma$) restricts to conjugation by $g_j$ (respectively $h_k$) on the stabilizer $G_f$ of $f$.

Moreover, we claim that $\sigma(g_j)= h_0 g_j h_0^{-1}$ and $\rho(h_k) = g_0 h_k g_0^{-1}$. If $\rho$ is a 
vertex automorphism associated to a vertex $v$, we have by Lemma \ref{VertexAMOnBassSerre} that $g_j$ is an element of 
$G_v$ on which $\sigma$ restricts to conjugation by $h_0$. If $\rho$ is a Dehn twist about 
$e = \Supp(\rho)$ the element $g_j$ is in $C(G_e)$: since $C(G_e)$ is cyclic by hypothesis, and since $\sigma$ restricts to conjugation by $h_0$ on $G_e$, 
it must also send $g_j$ to $h_0 g_j h^{-1}_0$.

Let $u$ be a vertex of $T^0$ which lies in $R^0 \cap S^k$. On $G_u$, we see that $\rho \circ \sigma$ restricts to conjugation by $\rho(h_k)g_0 = g_0 h_k$. Similarly $\sigma \circ \rho$ restricts to conjugation 
by $\sigma(g_0) h_k = h_0 g_0 h_0^{-1} h_k$. Thus $\rho \circ \sigma$ restricts to 
$\Conj (g_0 h_0 g_0^{-1} h_0^{-1}) \circ \sigma \circ \rho$ on $G_u$. The case for $u$ in $R^j \cap S^0$ is symmetric.

If $u$ is the support of one of the two elementary automorphisms, without loss of generality $\sigma$, 
the restriction of $\sigma \circ \rho$ on $G_u$ is $\Conj(\sigma(g_0)) \circ \sigma\mid_{G_u}$, 
while the restriction of $\rho \circ \sigma$ on $G_u$ is $\Conj(g_0) \circ \sigma\mid_{G_u}$. 
Thus for any vertex $u$ of $T^0$, $\rho \circ \sigma$ restricts to $\Conj (g_0 h_0 g_0^{-1} h_0^{-1}) \circ \sigma \circ \rho$ on $G_u$.

Let now $e=uv'$ be an edge of $T^1 \setminus T^0$ such that $v'$ lies in $T^1$ but not in $T^0$, and $v = t_e^{-1} \cdot v'$ is in $T^0$. 
Suppose $u$ lies in $R^0 \cap S^k$ and $v$ in $R^j \cap S^0$. By Lemmas \ref{DehnTwistOnBassSerre} 
and \ref{VertexAMOnBassSerre}, we know that $\rho (t_e) = g_0 z t_e g_j^{-1}$ and $\sigma(t_e) = h_k w t_e h_0^{-1}$ 
for elements $z, w$ of $C(\Stab (f))$. Note that $\rho$ and $\sigma$ restrict on $C(\Stab (f))$ to 
conjugation by $g_0$ and $h_k$ respectively. From this we see that 
$\rho \circ \sigma(t_e)= \Conj (g_0 h_0 g_0^{-1} h_0^{-1}) \circ \sigma \circ \rho (t_e)$. 

The cases remaining (when both $u$ and $v$ lie in $R^0 \cap S^k$, and where one or both of $u$, $v$, coincide with the support of $\rho$ or $\sigma$) are dealt with in a similar way to conclude the proof.
\end{proof}

\subsection{Modular groups}

Let $G$ be a torsion-free hyperbolic group which is freely indecomposable with respect to a subgroup $H$. 
As in \cite{PerinSklinosHomogeneity}, we define the relative modular group $\Mod_H(G)$ as the subgroup of 
$\Aut_H(G)$ generated by Dehn twists about one-edge cyclic splittings of $G$ in which $H$ is elliptic.

Recall that we have the following result \cite[Corollary 4.4]{RipsSelaHypI}:

\begin{thm} \label{} 
Let $G$ be a torsion-free hyperbolic group freely indecomposable with respect to a (possibly trivial) subgroup $H$.
The modular group $\Mod_H(G)$ has finite index in $\Aut_H(G)$.
\end{thm}

We now relate the cyclic JSJ decompositions and modular groups. This will enable us to give a ``normal form'' for modular automorphisms. 

First we define a group of automorphisms associated to a $G$-tree.
\begin{defi} Let $G$ be a finitely generated group, let $H$ be a subgroup of $G$, and let 
$T$ be a cyclic $(G, H)$-tree with a distinguished set of orbits of vertices which are of surface type.

The group of elementary automorphisms of $G$ with respect to $T$, $\Aut^T_H(G)$, is the subgroup of 
$\Aut_H(G)$ generated by Dehn twists about edges of $T$, vertex automorphisms associated to surface type vertices, 
and inner automorphisms.
\end{defi}

\begin{lemma}[Normal Form Lemma] \label{NormalFormMod} Let $T$ be a cyclic $(G, H)$-tree whose edge stabilizers have cyclic centralizers. 
Let $(T^1, T^0, (t_f)_{f})$ be a Bass-Serre presentation for $G$ with respect to $T$: 
any element $\theta$ of $\Aut^T_H(G)$ can be written as a product of the form 
$$\Conj(z) \circ \rho_1 \circ \ldots \circ \rho_r $$
where the $\rho_j$ are Dehn twists about distinct edges of $T^1$ or vertex automorphisms associated to distinct surface type vertices of $T^0$. 
Moreover, we can permute the list of supports of the $\rho_j$.

Finally, if $H$ fixes a non surface type vertex $x$ of $T^0$, we can in fact choose the $\rho_j$ to fix $Stab_G(x)$ pointwise, and thus $z$ to lie in the centralizer of $H$.
\end{lemma}

\begin{proof}  This follows easily from Lemma \ref{ElementaryAMOnTranslatedSupp} and Proposition \ref{ElementaryAMCommute}. 
The last statement follows from Remarks \ref{DehnTwistFixingAGivenVertex} and \ref{VertexAMFixingAGivenVertex}.
\end{proof}

The universal properties of the JSJ imply the following result, which can be seen as a special case of \cite[Theorem 5.4]{GuirardelLevittSplittingsAndAM}.
\begin{prop} \label{EquivalentDefOfMod} 
Let $G$ be a torsion-free hyperbolic group and let $H$ be a subgroup of $G$ with respect to which $G$ is freely indecomposable. 
Let $T$ be the pointed cyclic JSJ tree of $G$ with respect to $H$. Then $\Aut^T_H(G) = \Mod_H(G)$.
\end{prop}

To prove it, we will use the following lemmas, which relate elementary automorphisms associated to $G$-trees $\hat{T}$ and $T$ when 
$\hat{T}$ is a refinement of $T$. The proof of the first of these results is immediate.

\begin{lemma} \label{DehnTwistThroughCollapseUp} Let $\hat{T}$ and $T$ be two $G$-trees and suppose $p: \hat{T} \to T$ is a collapse map. 
Let $\tau$ be a Dehn twist by an element $a$ about an edge $e$ of $T$. 

Then $\tau$ is the Dehn twist by $a$ about the unique edge $\hat{e}$ such that $p(\hat{e})=e$.
\end{lemma}

\begin{lemma} \label{DehnTwistsThroughCollapseDown} Let $\hat{T}$ and $T$ be two $G$-trees and suppose $p: \hat{T} \to T$ is a collapse map. 
Let $\hat{\tau}$ be a Dehn twist by an element $a$ about an edge $\hat{e}$ of $\hat{T}$. 

If $p(\hat{e})$ is an edge, $\hat{\tau}$ is a Dehn twist by $a$ about $p(\hat{e})$.
If $p(\hat{e})$ is a vertex $v$, and if  $a$ is in $\Stab(\hat{e})$, then $\hat{\tau}$ is a vertex automorphism associated to $v$. 
Its restriction to $G_v$ is the Dehn twist by $a$ about the edge $\hat{e}$ of the minimal subtree $\hat{T}_v$ of $\Stab(v)$ in $\hat{T}$.
\end{lemma}

\begin{proof}  
If $p(\hat{e})$ is an edge $e$, it is easy to see that the one edge splitting induced by 
$\hat{e}$ and by $p(\hat{e})$ are the same, which proves the claim. Suppose thus that $p(\hat{e})$ is a vertex $v$. 

We choose a Bass-Serre presentation $(\hat{T}^1, \hat{T}^0, (t_f)_{f})$ for $G$ with respect to $\hat{T}$ such that 
\begin{itemize}
 \item $\hat{e}$ is in $\hat{T^1}$;
 \item if  $T^i = p(\hat{T}^i)$ for $i =0,1$, the triple $(T^1, T^0, (t_f)_{f})$ is a Bass-Serre presentation for $G$ with respect to $T$
\end{itemize}
(this can be done by taking for $T^0$ the lift of a maximal subtree of $\hat{T}/G$ which contains a maximal 
subtree of each of the maximal subgraphs collapsed under $p$).

The minimal subtree $\hat{T}_v$  of $G_v$ in $\hat{T}$ is contained in the preimage of $v$ by $p$. 
In particular, any translate of $\hat{T}_v$ by an element of $G \setminus G_v$ is disjoint from $\hat{T}_v$, 
so two vertices (respectively two edges) of $\hat{T}_v$ are in the same orbit under $G_v$ if and only 
if they are in the same orbit under the action of $G$. By our choice of Bass-Serre presentation, 
$(\hat{T}^1 \cap \hat{T}_v, \hat{T}^0  \cap \hat{T}_v, (t_f)_{f \in E((\hat{T}^{1} \setminus \hat{T}^0)  \cap \hat{T}_v)})$ 
is a Bass-Serre presentation for $G_v$ with respect to $\hat{T}_v$. 

From this it is easy to check that the restriction of $\hat{\tau}$ to $G_v$ is the 
Dehn twist by $a$ about $\hat{e}$ with respect to the action of $G_v$ on $\hat{T}_v$. 

Consider now the map $\pi: T \to \pi(T)$ which collapses all the orbit of edges of $T$ which are not adjacent to $v$. 
Let $f = \pi(v) w$ be an edge adjacent to $\pi(v)$: there is a unique edge $\hat{f}$ of $T$ such that $\pi \circ p (\hat{f}) = f$. 
Then $\hat{\tau}$ restricts on $\Stab(\hat{f})$ to a conjugation by an element which is either $1$ or $a$: both fix $\pi(v)$. 
If $w$ is not in the orbit of $\pi(v)$, the subtree $\pi^{-1}(w)$ of  
$T$ is stabilized by $G_w$ and does not meet any translates of $e$. 
By Lemma \ref{DisjointDehnTwistFixes}, $\hat{\tau}$ restricts on $G_w$ to a conjugation 
by $1$ or by $a$, which both fix $\pi(v)$. Thus $\hat{\tau}$ is a vertex automorphism with respect to $\pi(v)$. 
\end{proof}

We show:

\begin{lemma} \label{ModTreeInMod} 
Let $G$ be a torsion-free hyperbolic group which is freely 
indecomposable with respect to a subgroup $H$. Suppose $T$ is a cyclic $(G,H)$-tree with a 
distinguished set of orbits of vertices which are of surface type. 
Then $\Aut^T_H(G)$ is a subgroup of $\Mod_H(G)$.
\end{lemma}

\begin{proof} 
Suppose $\tau$ is a Dehn twist about an edge $e$ of $T$ which fixes $H$ pointwise. 
By definition, $\tau$ is the Dehn twist by $a$ associated to the one edge splitting obtained from $T$ 
by collapsing all the edges which are not in the orbit of $e$, thus it lies in $\Mod_H(G)$.

Now let $\sigma$ be a vertex type automorphism associated to a surface type vertex $v$ of $T$, with corresponding surface $\Sigma$. 
It is a classical result that the group of automorphisms of the fundamental group of a surface is generated by 
Dehn twists $\delta_c$ by elements $c$ corresponding to simple closed curves $\gamma$ on the surface. 
Thus it is enough to show the result for a vertex type automorphism $\sigma$ which restrict to a Dehn twist $\delta_{c}$ on $G_v$.

Denote by $T^+$ the refinement of $T$ obtained by refining $v$ by the $G_v$-tree dual to the curve $\gamma$ on $\Sigma$, 
by $e^+$ the edge of $T^+$ stabilized by $c$. By Lemma \ref{DehnTwistsThroughCollapseDown}, the Dehn twist $\tau^+$ 
by $c$ about $e^+$ is an elementary automorphism associated to $T$ with support $v$, whose restriction to $G_v$ 
is exactly $\delta_c$. By Remark \ref{VertexAMDifferByDT}, $\sigma$ and $\tau^+$ differ by a product of Dehn twists about edges of $T$. 
These Dehn twists as well as $\tau^+$ are all elements of $\Mod_H(G)$ by the first part of the proof, thus so is $\sigma$.
\end{proof}

The proof of Proposition \ref{EquivalentDefOfMod} is now straightforward.
\begin{proof} 
By Lemma \ref{ModTreeInMod}, we have the inclusion $\Aut^T_H(G) \leq \Mod_H(G)$.

Conversely, let $T'$ be a cyclic $G$-tree with a unique orbit of edges in which $H$ is elliptic. 
Let $a$ be an element in the centralizer of the stabilizer of some edge $e$, and denote by $\tau$ the Dehn twist about $e$ by $a$. 

As noted in Remark \ref{AcylindricalAndUnivCompatible}, $T$ is universally compatible, so it admits a refinement $\hat{T}$ which collapses onto
$T'$ via some map $p:\hat{T} \to T'$. By Lemma \ref{DehnTwistThroughCollapseUp}, 
$\tau$ is a Dehn twist about an edge $\hat{e}$ of $\hat{T}$. Note that if $p(\hat{e})$ is a vertex, then it must be a surface type vertex so in particular the stabilizer of $\hat{e}$ is maximal cyclic. By Lemma \ref{DehnTwistsThroughCollapseDown},
this Dehn twist is an elementary automorphism associated to $T$. Hence $\tau$ is in $\Aut_H^T(G)$.
\end{proof}

We also observe:

\begin{lemma}\label{ElemMOD}
Let $A\leq H$ be subgroups of a torsion-free hyperbolic group $G$ which is freely indecomposable with respect to $A$. Let $T$ be a cyclic $(G,A)$-tree with a distinguished set of orbits of vertices which are of surface type. Suppose $\rho$ is an elementary 
automorphism associated to $T$ whose support does not lie in any translate of the minimal subtree $T_H$ of $H$ in $T$. Then there exists $g\in G$ and $\sigma\in Mod_H(G)$ such that $\rho=Conj(g)\circ\sigma$. 
\end{lemma}
\begin{proof}
Consider the tree $T'$ obtained from $T$ by collapsing each subtree in the orbit of $T_H$. It is a $(G,H)$-tree 
with a distinguished set of orbits of vertices which are of surface type (inherited from $T$) 
and $\rho$ is an elementary automorphism associated to $T'$. By Remarks 
\ref{DehnTwistFixingAGivenVertex} and \ref{VertexAMFixingAGivenVertex} we have that 
there is $g\in G$ and an elementary automorphism $\sigma$ associated to $T'$ such that 
$\rho=Conj(g)\circ\sigma$ and $\sigma$ fixes $H$ pointwise. 
By Lemma \ref{ModTreeInMod}, $\sigma\in\Mod_H(G)$. 
\end{proof}

The following result can be seen as a generalization of Lemma \ref{NormalFormMod}.

\begin{prop} \label{NormalFormOfAMFixingH} 
Let $G$ be a torsion-free hyperbolic group, freely indecomposable with respect to a subgroup $A$. 
Let $T$ be the pointed JSJ tree of $G$ with respect to $A$. 
Let $H$ be a finitely generated subgroup of $G$ containing $A$, and denote by $T_{H}$ the minimal subtree of $H$ in $T$.

If $\theta$ is an element of $\Mod_H(G)$, it can be written as 
$$ \theta = \Conj(z) \circ \tau_{e_1} \circ \ldots \circ \tau_{e_p} \circ \sigma_{v_1} \circ \ldots \circ \sigma_{v_q} $$
where each $\tau_{e_i}$ is a Dehn twist about an edge $e_i$ of $T$ and each  $\sigma_{v_j}$ is a 
vertex automorphism $\sigma_{v_j}$ associated with a flexible vertex $v_j$ of $T$ such that:
\begin{itemize}
 \item the edge $e_i$ does not lie in any translate of $T_H$ for any $i$,
 \item if $v_j$ lies in some translate of $T_H$ then the restriction of $\sigma_{v_j}$ 
 to the corresponding surface group fixes an element representing a non-boundary parallel simple closed curve.
\end{itemize}

\end{prop}

\begin{proof}  By definition of the modular group, and by Proposition \ref{ElementaryAMCommute}, 
it is enough to prove the result when $\theta = \tau$ is the Dehn twist by some element $a$ about an edge 
$e$ of a cyclic $(G, H)$-tree $T'$ which has a unique orbit of edges. Since $G$ is torsion-free hyperbolic, 
we may further assume that the stabilizer of $e$ is maximal cyclic so that the tree $T'$ is $1$-acylindrical. 

Note that $T'$ is in particular a cyclic $(G,A)$-tree: by universal compatibility of the pointed JSJ tree, 
$T$ admits a refinement $\hat{T}$ which collapses onto $T'$. We thus have collapse maps $p: \hat{T} \to T$ and $p': \hat{T} \to T'$. 
The tree $\hat{T}$ is obtained by refining each surface type vertex $v$ by the minimal subtree 
$T_v$ in $T'$ of its stabilizer $G_v$.

Let $x$ be a vertex of $T'$ stabilized by $H$. Let $\hat{x}$ be a vertex of $\hat{T}$ such that $p'(\hat{x}) = x$. 
Denote by $\hat{T}_H$ the minimal subtree of $H$ in $\hat{T}$: it is covered by translates 
of paths $[\hat{x}, h \cdot \hat{x}]$. We have $p'(h \cdot \hat{x}) = x$ for all $h \in H$, so we see that $p'(\hat{T}_H) = \{x\}$. 

By Lemma \ref{DehnTwistThroughCollapseUp}, $\tau$ is a Dehn twist by $a$ about an edge $\hat{e}$ such that $p'(\hat{e}) = e$. 
Note that $\hat{e}$ does not lie in $\hat{T}_H$ since the images by $p$ of $\hat{T}_H$ is a single vertex. 

By Lemma \ref{DehnTwistsThroughCollapseDown}, if $p(\hat{e})$ is an edge, $\tau$ is a 
Dehn twist about $p(\hat{e})$. Also, if $p(\hat{e})$ is an edge, it lies outside of $p(\hat{T}_H)$ which contains $T_H$, so it lies outside of $T_H$. 

If $p(\hat{e})$ is a vertex, it must be a surface type vertex and $\Stab(e)$ is generated 
by an element corresponding to a simple closed curve on the corresponding surface: in particular 
it is maximal cyclic and $a$ is in $\Stab(e)$. Thus by Lemma \ref{DehnTwistsThroughCollapseDown}, 
$\tau$ is a vertex automorphism associated to the vertex $p(\hat{e})$ of $T$. 
Moreover, the restriction of $\tau$ to $\Stab(v)$ is a Dehn twist about an 
edge of the minimal tree of $\Stab(v)$ in $\hat{T}$ which is dual to a set non boundary parallel simple closed curves on the corresponding surface: this finishes the proof.
\end{proof}

\section{Isolating types} \label{IsolatingSec}

In this section, we show that if $G$ is a torsion-free hyperbolic group which is freely indecomposable with respect to $A$ 
and does not admit a structure of an extended hyperbolic 
tower over $A$, then the orbits of tuples of elements of 
$G$ under $\Aut_A(G)$ are definable over $A$ (equivalently, $G$ is atomic over $A$).

For the notion of an extended hyperbolic floor we refer the reader to \cite{LouderPerinSklinosTowers}. We give the following definition
\begin{defi}
Let $G$ be a torsion-free hyperbolic group and let $A\subset G$. Then  
$G$ is concrete with respect to $A$ if:
\begin{itemize}
 \item[(i)] $G$ is freely indecomposable with respect to $A$;
 \item[(ii)] $G$ does not admit the structure of an extended hyperbolic floor over $A$.
\end{itemize}
\end{defi}

Lemma 5.19 of \cite{PerinElementary} shows that if $G=\F$ is a free group, and if $A$ is not contained in a free factor, then $\F$ is concrete with respect to $A$. 

Here are some further examples where $G$ is concrete with respect to $A$:
\begin{ex}
\begin{itemize}
 \item[(i)] $G$ is the fundamental group of the connected sum of four projective planes, and $A$ is any 
 set of parameters that contain a non cyclic subgroup (see \cite[Lemma 3.12]{LouderPerinSklinosTowers}).
 \item[(ii)] $G:=\langle a,b,c,d | [a,b]=[c,d]\rangle$ is the fundamental group of the connected sum of two tori, and $A$ is any set of parameters that contain a subgroup of the form $\langle a,b,g\rangle$ with 
 $g\not\in\langle a,b\rangle$ (see the proof of \cite[Lemma 6.1]{LouderPerinSklinosTowers}).
\end{itemize}
\end{ex}

The proof of the result below is essentially contained in \cite{PerinSklinosHomogeneity}, but we include it here for reference. 
For the special case of free groups see also \cite[Proposition 5.9]{OuldHoucineHomogeneity}.

\begin{thm} \label{TypesIsolated} Let $G$ be a torsion-free hyperbolic group. Suppose $G$ is concrete with respect 
to $A$. Then for any $\bar{b}\in G$, the orbit of $\bar{b}$ under $\Aut_A(G)$ is definable over $A$, so in particular $tp^{G}(\bar{b}/A)$ is isolated. 
\end{thm}

\begin{proof} By Lemma 3.7 in \cite{PerinSklinosHomogeneity}, there is a finite subset $A_0$ of $A$ such that $G$ is freely 
indecomposable with respect to $A_0$, $A$ is elliptic in any $G$-tree in which $A_0$ is elliptic, and 
$\Mod_{A_0}(G) = \Mod_{A}(G)$. By Corollary 4.5 in \cite{PerinSklinosHomogeneity}, we can assume moreover 
that any embedding $j: G \to G$ which restricts to the identity on $A_0$ restricts to the identity on $A$.

By Theorem 4.4 in \cite{PerinSklinosHomogeneity}, there is a finite set of quotients $\{\eta_j: G \to Q_j\}_{j=1, \ldots, m}$ 
such that any non injective endomorphism $\theta: G \to G$ which restricts to the identity on $A_0$ factors through one 
of the quotient maps $\eta_j$ after precomposition by an element $\sigma$ of $\Mod_A(G)$. 
For each $j$, choose $u_j$ a non trivial element in $\Ker(\eta_j)$. 

Let $\gamma_1, \ldots , \gamma_r$ be a generating set for $G$. Write each element $a$ of $A_0$, each element $u_j$, and the tuple $\bar{b}$ as a 
word $w_a(\gamma_1, \ldots, \gamma_r)$ (respectively $w_{u_j}(\gamma_1, \ldots, \gamma_r)$, respectively a tuple 
$\bar{w}_{\bar{b}}(\gamma_1, \ldots, \gamma_r)$).

Let $\Lambda$ be a JSJ decomposition of $G$ with respect to $A$. Two endomorphisms $h$ and $h'$ of $\F$ are said to be 
$\Lambda$-related if $h$ and $h'$ coincide up to conjugation on rigid vertex groups of $\Lambda$, and for any flexible 
vertex group $S$ of $\Lambda$, $h(S)$ is non abelian if and only if $h'(S)$ is non abelian. It is easy to see that there is a 
formula  $\Rel(\bar{x}, \bar{y})$ such that for any pair of endomorphisms $h$ and $h'$ of $G$, the morphism 
$h'$ is $\Lambda$-related to $h$ if and only if $G \models \Rel(h(\gamma_1, \ldots, \gamma_r), h'(\gamma_1, \ldots, \gamma_r))$ 
(see Lemma 5.18 in \cite{ThesisPerin}).

Consider now the following formula $\phi(\bar{z}, A_0)$:
\begin{eqnarray*}
 \exists x_1, \ldots, x_r \; && \left\{ \bar{z} = \bar{w}_{\bar{b}}(x_1, \ldots, x_r) \wedge \bigwedge_{a \in A_0} a=w_a(x_1, \ldots, x_r) \right\} \\
&& \wedge \; \forall y_1, \ldots, y_r \left \{ \Rel(\bar{x},\bar{y}) \rightarrow \bigvee_j w_{u_j}(y_1, \ldots, y_r) \neq 1 \right\}.
\end{eqnarray*}

Suppose $G \models \phi(\bar{c}, A_0)$. Then the endomorphism $h: G \to G$ given by $\gamma_j \mapsto x_j$ sends 
$\bar{b}$ to $\bar{c}$ and fixes $A_0$, moreover no endomorphism $h'$ which is $\Lambda$-related to $h$ factors through one of the maps $\eta_i$. 
This implies that $h$ is injective, but by the relative co-Hopf property for torsion-free hyperbolic groups 
(see Corollary 4.2 of \cite{PerinSklinosHomogeneity}), this in turn implies that $h$ is an automorphism fixing $A_0$. 
By our choice of $A_0$, we have in fact $h \in \Aut_A(G)$. 
Thus the set defined by the formula $\phi(\bar{z}, A_0)$ is contained in the orbit of $\bar{b}$ by $\Aut_A(G)$.

To finish the proof is enough to show that $G\models \phi(\bar{b},A_0)$. 

It is obvious that the first part 
of the sentence $\phi(\bar{b})$ is satisfied by $G$ (just take $x_j$ to be $\gamma_j$). If the second part is not satisfied, this 
means that there exists an endomorphism $h': G \to G$ which is $\Lambda$-related to the identity, and which kills 
one of the elements $u_j$. Thus $h'$ restricts to conjugation on the rigid vertex groups of $\Lambda$, 
sends surface type flexible vertex groups on non abelian images, and is non injective: 
it is a non injective preretraction $G \to G$ with respect to $\Lambda$. By Proposition 5.11 in \cite{PerinElementary}, 
this implies that $G$ admits a structure of hyperbolic floor over $A_0$, thus over $A$, a contradiction. 
\end{proof}

We further remark that in the case where the subset $A\subset G$ is not contained in any proper retract of $G$, then 
the isolating formula can be taken to be Diophantine, i.e. $\exists\bar{y}(\Sigma(\bar{x},\bar{y},\bar{a})=1)$. This follows 
easily from a recent result of Groves \cite{GroRet}:

\begin{thm}
Let $G$ be a torsion-free hyperbolic group. Suppose $A$ is not contained in any proper retract of $G$. Then 
any endomorphism of $G$ that fixes $A$ is an automorphism.
\end{thm}

\section{Algebraic closures} \label{AlgClosureSec}
As Lemma \ref{ForkAlg} shows, the notion of forking independence is preserved under taking algebraic closures of the triple under 
consideration. For instance in order to prove 
that two tuples $\bar{b}$, $\bar{c}$ fork over a set of parameters $A$, 
it is enough (by transitivity of forking and the above mentioned lemma) to show that some elements 
$b'$ and $c'$ in the respective algebraic closure $\acl(A\bar{b})$ and 
$\acl(A\bar{c})$ fork over $A$. Thus, it will be useful to understand $acl(A)$ for $A$ a subset of 
a torsion-free hyperbolic group $G$.

It is not hard to see that if the subgroup generated by $A$ is cyclic, 
the algebraic closure of $A$ is the maximal cyclic subgroup containing $A$ (see Lemma 3.1 in \cite{OuldHoucineVallinoAlgClosure}). 

If $G$ is concrete with respect to $A$, we can use the results of the previous section to get

\begin{prop}  \label{AlgClosure} 
Let $G$ be a torsion-free hyperbolic group which is concrete with respect to a subgroup $A$. 
Let $(T, v_A)$ be the pointed cyclic JSJ tree of $G$ with respect to $A$. 
Then $\Stab(v_A)$ is contained in the algebraic closure of $A$ in $G$.
\end{prop}

\begin{proof} Let $\bar{b}$ be a tuple in $\Stab(v_A)$. By Proposition \ref{EquivalentDefOfMod}, $\Mod_A(G)$ is 
generated by elementary automorphisms associated to $(T, v_A)$ which fix the vertex group $\Stab(v_A)$ pointwise, 
thus $\bar{b}$ is fixed by $\Mod_A(G)$.

 
Since $\Mod_A(G)$ has finite index in $\Aut_A(G)$, the orbit of $\bar{b}$ under $\Aut_A(G)$ is finite. 
But by Theorem \ref{TypesIsolated}, this orbit is definable over $A$, thus $\bar{b}$ is in $acl_A(G)$
\end{proof}

The converse to this result does not hold: 
there could be some roots of elements of $\Stab(v_A)$ which are not in $\Stab(v_A)$, 
yet in torsion-free hyperbolic groups, algebraic closures are closed under taking roots. 
But this is the only obstruction: this was proved by Ould Houcine and Vallino in 
the case of free groups \cite{OuldHoucineVallinoAlgClosure}, and extends easily to the case considered here.

We continue with an easy corollary of the above proposition.

\begin{cor}\label{FinGenAlg}
Let $G$ be a torsion-free hyperbolic group which is concrete with respect to a subgroup $A$.  
Then there is a finitely generated subgroup $A_0$ of $A$ such that $A\subseteq acl_G(A_0)$. In particular $acl_G(A)=acl_G(A_0)$.
\end{cor}

\begin{proof}
Take $A_0$ to be the finitely generated subgroup of $A$ given by Proposition 3.7 in \cite{PerinSklinosHomogeneity}.
It is not hard to see that $G$ is concrete with respect to $A_0$. Let $(T,v_{A_0})$ be the 
pointed cyclic JSJ tree of $G$ with respect to $A_0$. By Proposition \ref{AlgClosure} 
we have that $\Stab(v_{A_0})\subseteq acl_G(A_0)$, but $A$ fixes $v_{A_0}$, thus we get the result.
\end{proof}

We moreover prove:

\begin{prop} \label{MinSubtreeAndAlgClos} Let $G$ be a torsion-free hyperbolic group concrete with respect to a 
subgroup $A$, and let $H$ be a finitely generated non abelian subgroup of $G$ which contains $A$.
Let $(T,v_A)$ and $(T',v_H)$ be the pointed cyclic JSJ trees of $G$ relative to $A$ and $H$ respectively.
Denote by $T_{H}$ the minimal subtree of $H$ in $T$.
\begin{itemize}
 \item If $U$ is the non cyclic stabilizer of a rigid vertex of $T_{H}$, then $U$ is contained in $\Stab(v_H)$. 
 \item If $S$ is the stabilizer of a flexible vertex of $T_{H}$ with corresponding surface $\Sigma$, 
there is an element $\gamma$ of $S$ corresponding to a non boundary parallel simple closed curve on $\Sigma$ which is contained in $\Stab(v_H)$. 
\end{itemize}
In particular $U$ and $\gamma$ are fixed by $\Mod_H(G)$, and contained in $\acl_{G}(H)$.
\end{prop}

\begin{proof}The tree $T'$ is a cyclic 
tree in which $A$ is elliptic, thus the JSJ tree $T$ admits a refinement $\hat{T}$ which collapses onto $T'$.  We have collapse maps $p: \hat{T} \to T$ and $p': \hat{T} \to T'$. 

Denote by $\hat{T}_H$ the minimal subtree of $H$ in $\hat{T}$: we have $p(\hat{T}_H) = T_H$ and  $p'(\hat{T}_H) = \{v_H\}$. Let now $x$ be a vertex of $T_H$.

If $x$ is a rigid vertex of $T$ with non cyclic stabilizer $U$, $p^{-1}(x)$ is reduced to a point 
$\hat{x}$ which must lie in $\hat{T}_H$, and also has stabilizer $U$. Now $p'(\hat{x}) = v_H$ so $U$ lies in $\Stab(v_H) \subseteq \acl_G(H)$.

If $x$ is a surface type vertex of $T$, the action of $\Stab(x)$ on $p^{-1}(x)$ is dual to 
a set of disjoint simple closed curves on $\Sigma$, so the stabilizer of any vertex $\hat{x}$ in $p^{-1}(x) \cap \hat{T}_H$ corresponds to a subsurface of 
$\Sigma$ which is not an annulus parallel to a boundary. In particular, $\Stab (\hat{x})$ 
contains an element $\gamma$ corresponding to a non boundary parallel simple closed curve on $\Sigma$. Again $p'(\hat{x}) = v_H$ so $\Stab(\hat{x})$ lies in $\acl_G(H)$. In particular $\gamma$ is in $\Stab(v_H) \subseteq \acl_G(H)$. 
\end{proof}

We finish this section with a result which generalizes Proposition \ref{MinSubtreeAndAlgClos}:
\begin{prop}\label{RigidTranslates}
Let $G$ be a torsion-free hyperbolic group concrete with respect to a 
subgroup $A$, and let $H$ be a finitely generated non abelian subgroup of $G$ which contains $A$.
Let $(T,v_A)$ be the pointed cyclic JSJ tree of $G$ relative to $A$.
Denote by $T_{H}$ the minimal subtree of $H$ in $T$.

Let $v$ be a vertex of $T$ such that the path from $T_H$ to $v$ consists of edges which all lie in translates of $T_H$, and does not contain any surface type vertices.

Then $\Stab(v)\subseteq acl_G(H)$.
\end{prop}

\begin{proof} Let $(T',v_H)$ be the pointed cyclic JSJ tree of $G$ relative to $H$.

The JSJ tree $T$ admits a refinement $\hat{T}$ which collapses onto $T'$.  We have collapse maps $p: \hat{T} \to T$ and $p': \hat{T} \to T'$. Any non surface type vertex $y$ of $T$ (respectively any edge $e$ of $T$) has a preimage by $p$ a single vertex (respectively a unique edge), which we denote $\hat{y}$ (respectively $\hat{e}$). Moreover we have $\Stab(\hat{y}) = \Stab(y)$.

The hypotheses on $v$ imply that the path between $\hat{T}_H$ and $\hat{v}$ consists exactly of the lifts $\hat{e}$ of the edges $e$ of the path $[u,v]$ between $T_H$ and $v$. 

Now each $\hat{e}$ lies in a translate of the path $[\hat{v}_A, h \cdot \hat{v_A}]$ for some $h \in H$, and this path is collapsed under the map $p'$. Thus all the edges $\hat{e}$ are collapsed under $p'$ so $p'(\hat{v}) = p'(\hat{u})$.

Finally, as in the proof of Proposition \ref{MinSubtreeAndAlgClos}, we can see that for any vertex $y$ of $T_H$ which is of non surface type we have $p'(\hat{y}) = v_H$. Thus $p'(\hat{u}) = v_H$, and $\Stab(v) = \Stab(\hat{v})$ is contained in $\Stab(v_H) \subseteq \acl_G(H)$.
\end{proof}

\section{The curve complex} \label{CCSec}
In this section we give some basic definitions and results about the curve 
complex assigned to a surface introduced by Harvey \cite{HarveyCC}. These will be useful for the proof of Theorem \ref{FreelyIndec}.

\begin{defi}\label{CurveComplex}
Let $\Sigma$ be a surface with (possibly empty) boundary. 
Then the curve complex ${\cal{C}}(\Sigma)$ is the simplicial complex given by: 
\begin{itemize}
 \item[(i)] $0$-simplices are simple closed curves (up to free homotopy) on $\Sigma$ which do not bound a disk, an annulus, or a M\"obius band; 
 \item[(ii)] A subset $\{\gamma_0,\ldots,\gamma_k\}$ of the set of $0$-simplices forms a $k$-simplex if the curves in the subset can be realized disjointly.
\end{itemize}
\end{defi}

\begin{figure}[ht!]

\centering
\includegraphics[width=.9\textwidth]{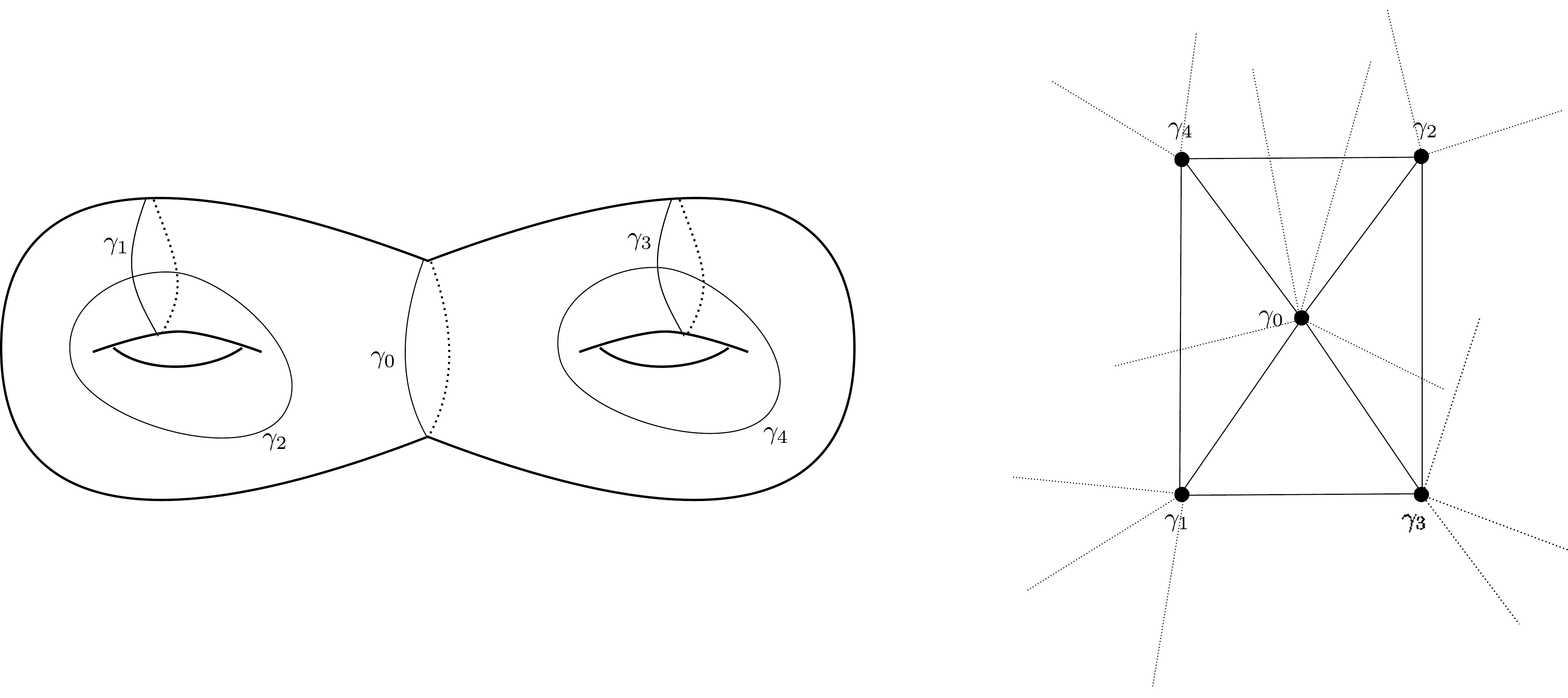}
\caption{Part of the curve complex of the orientable surface of genus $2$.}

\end{figure}

\begin{rmk}
Let $\Sigma_{g,n}$ be the orientable surface of genus $g$ with $n$ boundary components. 
In the following sporadic cases Definition \ref{CurveComplex} gives a degenerate discrete set or even the empty set.  
\begin{itemize}
 \item $\Sigma_{0,n}$ for $n\leq 4$;
 \item $\Sigma_{1,n}$ for $n\leq 1$;
 \item The $n$-punctured projective plane for $n \leq 2$;
 \item The $n$-punctured Klein bottle with $n \leq 1$. 
\end{itemize}
\end{rmk}

If $\Sigma$ is the once-punctured 
torus $\Sigma_{1,1}$ or the four-punctured sphere $\Sigma_{0,4}$ we modify the second part of the definition: a subset of the set of $0$-simplices forms a simplex if the corresponding curves can be realized with intersection number at most $1$ (respectively $2$ in the case of $\Sigma_{0,4}$). Note that in both cases 
the resulting simplicial complex is the well known Farey graph.


Combining Theorem 1.1 in \cite{MasurMinskyHypCC} and the results in the appendix of 
\cite{BestvinaFujiwaraQHomoOfMCG} we get

\begin{thm}
Let $\Sigma$ be a surface which is either a punctured torus or has Euler characteristic at most $-2$. 
Then ${\cal C}(\Sigma)$ has infinite diameter.
\end{thm}

The mapping class group ${\cal MCG}(\Sigma)$ of the surface $\Sigma$, that is, the group of isotopy classes of the 
self-homeomorphisms of $\Sigma$ (fixing each boundary component pointwise), acts on the curve complex of $\Sigma$ in the obvious way. 
We observe the following:

\begin{lemma} \label{DisjointBallsInCC} 
Let $\Sigma$ be a surface which is either a punctured torus or has Euler characteristic at most $-2$. Let $R \geq 0$. 
There exists a sequence of elements $h_n$ of ${\cal MCG}(\Sigma)$ such that for any vertex 
$x$ of ${\cal C}(\Sigma)$ the translates $h_n(B_R(x))$ of the ball of radius $R$ around $x$ by the $h_n$ are pairwise disjoint.
\end{lemma}
\begin{proof} 
It is immediate that there is only a finite number of orbits of vertices in ${\cal C}(\Sigma)$ under the action of ${\cal MCG}(\Sigma)$. 
Let $M$ be such that any ball of radius $M$ in ${\cal C}(\Sigma)$ meets each of these orbits. 

Since ${\cal C}(\Sigma)$ has infinite diameter, we can find a sequence $y_n$ of vertices such that 
$d(y_n, y_m) > 2(M+R)$ whenever $m \neq n$. 
By our choice of $M$, each of the balls $B_M(y_n)$ contains a vertex $x_n = h_n(x)$ in the orbit of $x$. 
Thus the balls $B_R(x_n)$ are pairwise disjoint. 
\end{proof}

Before proving our next lemma we recall the correspondence 
between the geometric notions mentioned above and their algebraic counterparts.   

We fix a surface with a basepoint $(\Sigma,*)$. We note that the free homotopy class of a simple 
closed curve $\alpha$ on $\Sigma$ corresponds to the conjugacy class $[a]$ of an element $a$ representing 
$\alpha$ in $S:=\pi_1(\Sigma,*)$.

Moreover, a mapping class $h$ in ${\cal MCG}(\Sigma)$ gives rise to an outer automorphism of the fundamental group 
$S$ as a surface group with boundary (that is, an outer automorphism that fixes the conjugacy 
classes corresponding to the boundary components). It is a classical result that this induces an isomorphism between 
${\cal MCG}(\Sigma)$ and $\Out(S)$. So, we have:

\begin{lemma}\label{sccfork}
Let $\Sigma$ be a surface which is either a punctured torus or has Euler characteristic at most $-2$. 
Let $[a],[b]$ be conjugacy classes representing simple closed curves $\alpha,\beta$ in $\Sigma$. 
Then there is a sequence $(\rho_n)_{n<\omega}\in Out(S)$, such that for any $f_1,f_2\in Out_{[b]}(S)$ (i.e. outer 
automorphisms fixing the conjugacy class of $b$) and $i\neq j$, $\rho_i\circ f_1([a])\neq \rho_j\circ f_2([a])$.  
\end{lemma}
\begin{proof}
We apply Lemma \ref{DisjointBallsInCC} for $R=d_{{\cal C}(\Sigma)}(\alpha,\beta)$. It is a straightforward
excercise to confirm that the sequence of outer automorphisms $(\rho_n)_{n < \omega}$ corresponding to the sequence of mapping 
classes $(h_n)_{n < \omega}$ given 
by Lemma \ref{DisjointBallsInCC} satisfies the conclusion. 
\end{proof}




\section{Forking over big sets} \label{FreelyIndecSec}
In this section we bring results from previous sections together in order to prove Theorem \ref{FreelyIndecIntro}. 

We start with a lemma that connects forking independence with the modular group of a torsion-free hyperbolic group 
concrete with respect to a set of parameters. 
 
\begin{lemma} \label{WorkInStandardModel} 
Let $G$ be a torsion-free hyperbolic group, and let $A$ be a subset of $G$ with respect to which $G$ is concrete. 
Let $\bar{b},\bar{c}$ be tuples in $G$.
Suppose that the orbit $\Mod_{A\bar{c}}(G).\bar{b}$ is preserved by $\Mod_A(G)$. Then $\bar{b} \underset{A}{\forkindep} \bar{c}$. 
\end{lemma}

\begin{proof} 
Let $X:=\Aut_{A\bar{c}}(G).\bar{b}$. By Proposition \ref{TypesIsolated}, $X$ is definable over $A\bar{c}$. 
By Remark \ref{ForkImpl}, since 
$X$ implies every other formula in $tp(\bar{b}/A\bar{c})$ it is enough to prove that $X$ does not fork over $A$.

Now $\Mod_{A\bar{c}}(G).\bar{b}$ is a nonempty subset of $X$ preserved by $\Mod_A(G)$: since $\Mod_A(G)$ has finite index in $\Aut_A(G)$, 
this subset is almost $A$-invariant: by Lemma \ref{AtoFork}, we get the result.  
\end{proof}

%


We can now state and prove the second main result of the paper. 

\begin{thm} \label{FreelyIndec} Let $G$ be a torsion-free hyperbolic group, and 
let $A$ be a subset of $G$ with respect to which $G$ is concrete.

Let $(\Lambda, v_A)$ be the pointed cyclic JSJ decomposition of $G$ with respect to $A$. Let $\bar{b}$ and $\bar{c}$ be tuples of $G$, 
and denote by $\Lambda_{A\bar{b}}$ (respectively $\Lambda_{A\bar{c}}$) the minimal subgraphs of groups of 
$\Lambda$ whose fundamental group contains the subgroups $\langle A, \bar{b}\rangle$ (respectively $\langle A, \bar{c} \rangle$) of $G$.

Then $\bar{b}$ and $\bar{c}$ are independent over $A$ if and only if each connected component of 
$\Lambda_{A\bar{b}} \cap \Lambda_{A\bar{c}}$ contains at most one non Z-type vertex, and such a vertex is of 
non surface type.
\end{thm}

Note that since free groups are concrete over any set of parameters with respect to which they are freely indecomposable, 
Theorem \ref{FreelyIndecIntro} is a corollary of this result. 

\begin{rmk} We note that by Corollary \ref{FinGenAlg} coupled with 
Lemma \ref{ForkAlg}, there exists a finitely generated subgroup $A_0$ of 
$A$ such that tuples $\bar{b}$ and $\bar{c}$ fork over $A$ if and only if they fork over $A_0$. 
This easily extends to any finitely generated subgroup of $A$ which is "sufficiently large" (i.e. which contains $A_0$) .

On the other hand, by Proposition 3.6 in \cite{PerinSklinosHomogeneity}, 
there exists a finitely generated subgroup $A_0$ of $A$ such that the minimal subgraph 
of groups of $\Lambda$ whose fundamental group contains 
$\langle A, \bar{b} \rangle$ (respectively $\langle A, \bar{c} \rangle$) is the same as the the minimal 
subgraph of groups whose fundamental group contains $\langle A_0, \bar{b} \rangle$ 
(respectively $\langle A_0, \bar{c} \rangle$). Again this extends to any "sufficiently large" finitely generated subgroup of $A$.

Thus, in proving Theorem \ref{FreelyIndec},
we can always assume that the set of parameters is a finitely generated group.
\end{rmk}

We first prove the "if" direction.
\begin{lemma} 
Assume we are in the setting of Theorem \ref{FreelyIndec}.

Suppose that $\Lambda_{A\bar{b}} \cap \Lambda_{A\bar{c}}$ contains at most one 
non Z-type vertex, and such a vertex is of non surface type.
Then $\bar{b} \underset{A}{\forkindep} \bar{c}$. 
\end{lemma}

\begin{proof}Let $(T, v_A)$ be the pointed cyclic JSJ tree of $G$ with respect to $A$.
Let $T_{A\bar{b}}$ (respectively $T_{A\bar{c}}$) denote the minimal subtree of  the subgroup of 
$G$ generated by $A$ and $\bar{b}$ 
(respectively $A$ and $\bar{c}$) in $T$. 
Note that by definition $v_A$ lies in both $T_{A\bar{b}}$ and $T_{A\bar{c}}$.

Choose a Bass-Serre presentation 
$(T^1, T^0, (t_f)_f)$ for $G$ with respect to $\Lambda$ such that $v_A \in T^0$.

By Lemma \ref{WorkInStandardModel}, it is enough to show that $\Mod_{A\bar{c}}(G).\bar{b}$ 
is preserved by $\Mod_{A}(G)$. For this, it is enough to show that for any 
$\theta$ in $\Mod_A(G)$, we can find $\alpha$ in $\Mod_{A\bar{c}}(G)$ such that $\theta(\bar{b}) = \alpha(\bar{b})$. 

By Lemma \ref{NormalFormMod}, $\theta$ can be written as a product of the form
$$ \Conj(z) \circ \rho_1 \circ \ldots \circ \rho_t$$
where the $\rho_j$ fix $A$ pointwise, are Dehn twists about distinct edges of $T^1$ or 
vertex automorphisms associated to distinct surface type vertices of $T^0$. 

The hypothesis on $\Lambda_{A\bar{b}} \cap \Lambda_{A\bar{c}}$ implies that the intersection of 
$\bigcup_{g\in G} g \cdot T_{A\bar{b}} $ with $\bigcup_{h\in G} h \cdot T_{A\bar{c}} $ contains no surface 
type vertex. Also, it implies that this intersection meets at most one orbit of edge of each cylinder of $T$.

In this light, we may assume that the supports of the $\rho_j$ lie outside of 
$\bigcup_{g\in G} g \cdot T_{A\bar{b}} \cap \bigcup_{h\in G} h \cdot T_{A\bar{c}}$. 
Indeed, suppose that $\Supp(\rho_j)$ is in $g \cdot T_{A\bar{b}} \cap h \cdot T_{A\bar{c}}$: 
it must be an edge by the remark above. By Lemma \ref{RelationDehnTwists} it can be replaced by a 
product of a conjugation and Dehn twists whose supports are edges in the same cylinder which 
are not in the orbit of $e$: they must lie outside of $\bigcup_{g\in G} g \cdot T_{A\bar{b}} \cap \bigcup_{h\in G} h \cdot T_{A\bar{c}}$.

Since for each $j$ we have that either $\rho_j$ does not belong to any translate of $T_{A\bar{b}}$ or it 
does not belong to any translate of $T_{A\bar{c}}$, we may assume (using Lemma \ref{ElemMOD} and Proposition \ref{ElementaryAMCommute}) 
that there exists $r$ such that $\rho_i\in Mod_{A\bar{c}}$ for any $i\leq r$ and $\rho_j\in Mod_{A\bar{b}}$ for any $j>r$. 

Also observe that since $\theta$ and each $\rho_j$ fix $A$, we have that either $z$ is trivial or $A$ is cyclic and 
$z\in C(A)$. In the first case we can take $\alpha$ to be $\rho_1\circ\ldots\circ\rho_r$. 

In the second case we let $\tau$ be the product of the Dehn twists by $z$ about the edges of 
$T^1$ which are in the unique cylinder containing $v_A$, but do not lie in $T_{A\bar{c}}$. 
Then $\tau$ satisfies $\tau(\bar{b}) = \Conj(z)(\bar{b})$, and lies in $\Mod_{A\bar{c}}(G)$. 
Thus we can take $\alpha$ to be $\rho_1 \circ \ldots \circ \rho_r \circ  \tau$.
\end{proof}

To prove the second direction of Theorem \ref{FreelyIndec}, it 
is enough to consider the following three cases: $(i)$ for some $g$, the intersection $T_{A\bar{b}} \cap g \cdot T_{A\bar{c}}$ 
contains a surface type vertex, $(ii)$ for some $g, h, h'$, there are edges from distinct orbits $e=xz$ and $e'=yz$ contained in the intersection
$T_{A\bar{b}} \cap h \cdot T_{A\bar{c}}$ and $g \cdot T_{A\bar{b}} \cap h' \cdot T_{A\bar{c}}$ 
respectively, where each of $x$ and $y$ is either the basepoint, 
or non cyclically stabilized of rigid type, and $(iii)$ for some $g$, 
the intersection $T_{A\bar{b}} \cap g \cdot T_{A\bar{c}}$ contains an edge $e = v_Ax$ where $v_A$ is the basepoint, 
and $x$ is non cyclically stabilized of rigid type.

The following lemma deals with the latter case. Note that in this case, $A$ is cyclic. 
\begin{lemma} \label{IntersectInBasepointEdgeFork} Assume we are in the setting of Theorem \ref{FreelyIndec}. 
Let $T_{A\bar{b}}$ (respectively $T_{A\bar{c}}$) denote the minimal subtree of 
$\langle A, \bar{b} \rangle$ (respectively $\langle A, \bar{c} \rangle$) in $T$. 

Suppose that there exists an element $g$ such that the intersection $T_{A\bar{b}} \cap g \cdot T_{A\bar{c}}$ 
contains an edge $e = v_Ax$ where $v_A$ is the basepoint, and $x$ is a non 
cyclically stabilized vertexof rigid type.

Then $\bar{b}$ forks with $\bar{c}$ over $A$.
\end{lemma}

\begin{proof} 
By definition of the pointed cyclic JSJ tree, since $v_A$ is at distance $1$ of a non 
cyclically stabilized vertex we must have that $A$ is cyclic and $e$ is the unique edge adjacent to $v_A$. 
Now $v_A$ and $e$ are stabilized by the centralizer $C(A)$ of $A$, and the stabilizer of $x$ is not cyclic.

By Proposition \ref{MinSubtreeAndAlgClos} $\Stab(x)\subseteq acl_G(A\bar{b})$ and by Proposition \ref{RigidTranslates} 
$\Stab(x)\subseteq acl_G(A\bar{c})$, but $Stab(x)\not\subseteq acl_G(A)$. This implies that $\bar{b}$ forks with $\bar{c}$ over $A$. 

 

\end{proof}

We can now show that if we are in case $(ii)$, the tuples $\bar{b}$ and $\bar{c}$ fork over $A$.
\begin{lemma} \label{IntersectInEdgeFork} Assume we are in the setting of Theorem \ref{FreelyIndec}.
Let $T_{A\bar{b}}$ 
(respectively $T_{A\bar{c}}$) denote the minimal subtree of $\langle A, \bar{b} \rangle$ (respectively $\langle A, \bar{c} \rangle$) in $T$.  

Suppose that there exist elements $g, h, h'$ and edges from distinct orbits $e=xz$ and $e'=yz$ contained in the intersection
$T_{A\bar{b}} \cap h \cdot T_{A\bar{c}}$ and $g \cdot T_{A\bar{b}} \cap h' \cdot T_{A\bar{c}}$ respectively, where $x$ and $y$ are 
non cyclically stabilized vertices which are either the basepoint, or of rigid type.

Then $\bar{b}$ and $\bar{c}$ fork over $A$.
\end{lemma}

\begin{proof} Choose a Bass-Serre presentation $(T^1, T^0, (t_{\alpha})_{\alpha})$ for $G$ with respect to 
$\Lambda$ such that $e$ and $e'$ are in $T^1$.

\begin{figure}[ht!]

\centering
\includegraphics[width=.9\textwidth]{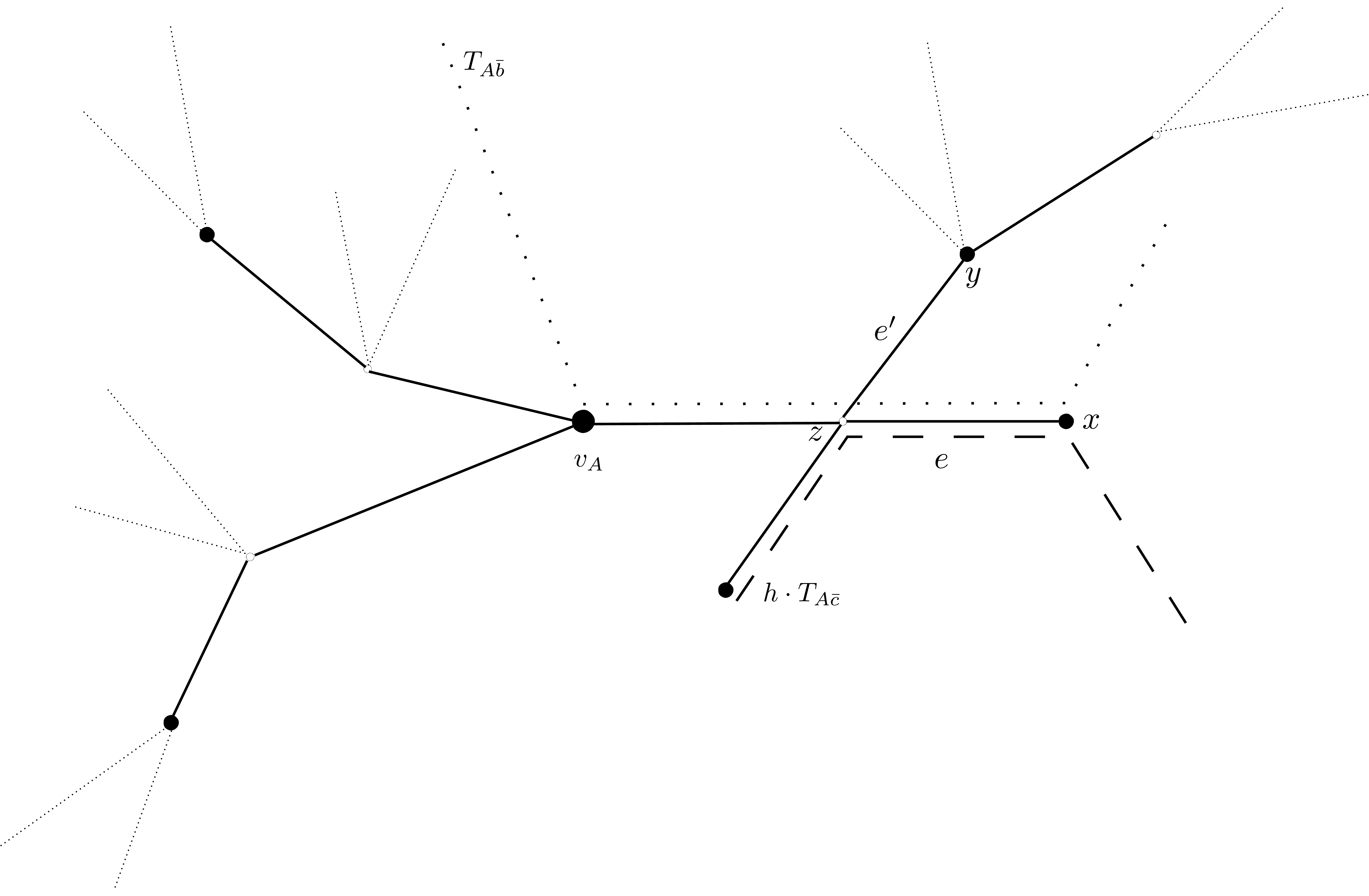}

\end{figure} 

Denote the stabilizers of $x$ and $y$ by $U$ and $V$ respectively, and let $\bar{u}, \bar{v}$ denote 
generating tuples of $U$ and $V$ respectively. We assume without loss of generality that $x \neq v_A$ so that $U$ is not cyclic.

Let $X$ denote the orbit of the pair $(\bar{u}, \bar{v})$ under $\Aut_{A\bar{c}}(G)$. 
We will show that $X$ admits infinitely many pairwise disjoint translates by a sequence of automorphisms in $\Aut_A(G)$: 
since $X$ is definable over $A\bar{c}$ and contains $(\bar{u},\bar{v})$, this will imply that $(\bar{u},\bar{v})$ forks with 
$\bar{c}$ over $A$. Now by Proposition \ref{MinSubtreeAndAlgClos} and \ref{RigidTranslates} 
respectively, both $\bar{u}$ and $\bar{v}$ are in $\acl(A\bar{b})$ so this implies that $\bar{b}$ forks with $\bar{c}$ over $A$. 

Let $\tau_e$ be a Dehn twist about $e$ by some element $\epsilon$ of $\Stab(e)$. 

By uniqueness of the tree $T$, for any element $\phi$ of $\Aut_A(G)$ there is an automorphism $f$ of 
$T$ such that for any $w \in T$ and $g \in G$ we have $f(g \cdot w) = \phi(g) \cdot f(w)$. 
Recall now that $\Mod_{A\bar{c}}(G)$ has finite index in $\Aut_{A\bar{c}}(G)$: pick $\phi_0, \ldots , \phi_l$ 
such that the classes $\Mod_{A\bar{c}}(G) \phi_j$ cover $\Aut_{A\bar{c}}(G)$, and denote by $f_1, \ldots, f_l$ 
the corresponding automorphisms of $T$. Since $\phi_j$ fixes $A\bar{c}$, the automorphism $f_j$ must preserve $T_{A\bar{c}}$. 
Now $e$ lies in $g \cdot T_{A\bar{c}}$, so the edges $f_j(e)$ also lie in a translate of $T_{A\bar{c}}$. 

Denote by $\tau_{f_k(e)}$ the automorphism $\phi_k \tau_e \phi^{-1}_k$: it is a Dehn twist about $f_k(e)$. 
Choose $j_1, \ldots, j_{l'}$ minimal such that $\{f_{j_1}(e), \ldots, f_{j_{l'}}(e)\}=\{f_1(e), \ldots, f_l(e)\}$ and define
$$ \tau = \tau_{f_{j_1}(e)} \circ \ldots \circ \tau_{f_{j_{l'}}(e)}.$$

We will show that for $r$ large enough the sequence of translates $\{ \tau^{rn}(X) \}_{n \in \N}$ consists of pairwise disjoint sets. 
For this it is enough to show that for $m$ large enough, $X \cap \tau^m(X)$ is empty. 
Suppose that there exists $j,k$ and $\alpha, \beta$ in $\Mod_{A\bar{c}}(G)$ such that
$$ \alpha (\phi_j (u, v)) = \tau^m \beta(\phi_k(u, v))$$

By Proposition \ref{NormalFormOfAMFixingH}, any element of $\Mod_{A\bar{c}}(G)$ can be written as a 
product of (a conjugation and) elementary automorphisms whose supports, if they are edges, 
are not in any translate of $T_{A\bar{c}}$ hence are not in the orbit of the edges $f_j(e)$. 
In particular $\beta$ commutes with $\tau$ (up to conjugation). 
Thus there exists $\theta= \beta^{-1} \alpha$ in $\Mod_{A\bar{c}}(G)$ such that $\theta (\phi_j (\bar{u}, \bar{v}))$ 
is conjugate to $\tau^m (\phi_k(\bar{u}, \bar{v}))$.

Since $e$ and $e'$ are not in the same orbit, we can choose a Bass-Serre presentation for $G$ with respect to $T$ such 
that $f_j(e)$ and $f_j(e')$ are in $T^1$. The automorphism $\theta$ can be written as a product of a conjugation and 
elementary automorphisms whose supports are not in the orbit of the edges $f_j(e)$ and $f_j(e')$. 
Thus Lemmas \ref{DehnTwistOnBassSerre} and \ref{VertexAMOnBassSerre} imply that $\theta (\phi_j (\bar{u},\bar{v}))$ is 
conjugate to the tuple $\phi_j (\bar{u},\bar{v})$. On the other hand, 
by definition of $\tau$ we have that $\tau^m (\phi_k(\bar{u},\bar{v}))$ is conjugate to $\phi_k(\bar{u}, \epsilon^m \bar{v} \epsilon^{-m})$ 
so finally there exists an element $\gamma$ such that
$$ \gamma \phi_j (\bar{u}, \bar{v}) \gamma^{-1} = \phi_k(\bar{u}, \epsilon^m \bar{v} \epsilon^{-m})$$

For $j$ and $k$ fixed, this holds for at most one value of $\gamma$, since we 
have $\gamma \phi_j(\bar{u}) \gamma^{-1} = \phi_k(\bar{u})$, and 
$\bar{u}$ generates a non abelian subgroup. 
But then, $\gamma \phi_j(\bar{v}) \gamma^{-1}= \phi_k(\epsilon^m \bar{v} \epsilon^{-m})$ 
can only be true for a single value of $m$ (for each $j, k$). Thus for $m$ large enough, $X \cap \tau^m(X)$ is empty.
\end{proof}

To finish the proof of Theorem \ref{FreelyIndec}, we need to deal with the case 
where translates of the minimal subtrees intersect in a surface type vertex. 
For this, we will use the results about the curve complex exposed in Section \ref{CCSec}. 

\begin{lemma}Assume we are in the setting of Theorem \ref{FreelyIndec}.
Let $T_{A\bar{b}}$ 
(respectively $T_{A\bar{c}}$) denote the minimal subtree of $\langle A, \bar{b} \rangle$ 
(respectively $\langle A, \bar{c} \rangle$) in $T$.  

Suppose that there exists an element $g$ of $G$ such that the intersection  $ (g^{-1} \cdot T_{A\bar{b}}) \cap T_{A\bar{c}}$ 
contains a surface type vertex. Then $\bar{b}$ and $\bar{c}$ fork over $A$.
\end{lemma}

\begin{proof} 
Denote by $v$ the surface type vertex, by $S$ its stabilizer and by 
$\Sigma$ the corresponding surface with boundary. 
Fix a Bass-Serre presentation $(T^1, T^0, (t_f)_f)$ such that $v$ lies in $T^0$. 
Denote by $v_{A\bar{b}}$ and $v_{A\bar{c}}$ the basepoints of the pointed cyclic 
JSJ trees of $G$ with respect to $\langle A, \bar{b} \rangle$ and $\langle A, \bar{c} \rangle$ respectively. 

By Proposition \ref{MinSubtreeAndAlgClos}, there exist elements $b_0$ and $c_0$ of $S$ 
which correspond to non boundary parallel simple closed curves on $\Sigma$, such that $b_0^g$ is 
contained in $\Stab(v_{A\bar{b}}) \subseteq \acl_{G}(A\bar{b})$, and $c_0$ is contained in 
$\Stab(v_{A\bar{c}}) \subseteq \acl_{G}(A\bar{c})$. 
%

Denote by $X$ the orbit of $b_0^g$ under $\Aut_{A\bar{c}}(G)$. 
We will show that $X$ admits infinitely many pairwise disjoint translates by 
a sequence of automorhisms in $\Aut_A(G)$: 
since $X$ is definable over $A\bar{c}$, this implies that $\bar{b}$ forks with $\bar{c}$ over $A$. 


As in the proof of Lemma \ref{IntersectInEdgeFork}, pick $\phi_1, \ldots , \phi_l$ 
such that the classes $\Mod_{A\bar{c}}(G) \phi_i$ cover $\Aut_{A\bar{c}}(G)$ and denote by $f_1, \ldots, f_l$ 
the corresponding automorphisms of $T$. Since $\phi_i$ fixes $A\bar{c}$, the automorphism $f_i$ must preserve $T_{A\bar{c}}$, 
hence the vertex $f_i(v)$ is a surface type vertex in $T_{A\bar{c}}$. 

Let $\{v_1, \ldots , v_s\}$ be the vertices of $T^0$ which lie in the orbit of one of the vertices $f_i(v)$. 
Up to reindexing we may assume $v_j$ is in the orbit of $f_j(v)$. 
Denote by $S_j$ the stabilizer of $v_j$, and by $b_j,c_j$ the images of $b_0, c_0$ by $\phi_j$.

Note that $\phi_j$ fixes $A\bar{c}$, so $c_j =\phi_j(c_0)$ is also in $\Stab(v_{A\bar{c}})$. 
Any element of $\Mod_{A\bar{c}}(G)$ fixes  $\Stab(v_{A\bar{c}})$ pointwise, so it preserves the conjugacy class of $c_j$.

By applying Lemma \ref{sccfork} to $\Sigma$ for the conjugacy classes of $b_0$ and $c_0$, we get a sequence of automorphisms 
$\rho^v_n$ of $\Aut(S)$ such that for any two automorphisms $ \sigma, \sigma'$ of $S$ 
which preserve the conjugacy class of $c_0$, if $m \neq n$ the elements $\rho^{v}_n\sigma(b_0)$ and 
$\rho^{v}_m \sigma'(b_0)$ are not conjugate in $S$ (and thus in $G$).

Let $\rho^{v_j}_n$ be a vertex automorphism associated to $v_j$ whose restriction to 
$S_j$ is $\phi_j \circ \rho^v_n \circ \phi^{-1}_j$  and define
$$\rho_n = \rho^{v_1}_n \ldots \rho^{v_s}_n.$$

We will show that $\rho_n(X) \cap \rho_m(X)$ is empty. Suppose not: then there exists $j,k$ and 
$\theta, \theta'$ in $\Mod_{A\bar{c}}(G)$ such that
\begin{eqnarray} \label{ConjRelation}
\rho_n \theta (\phi_j(b_0^g)) = \rho_m \theta'(\phi_k(b_0^g))
 \end{eqnarray}

By Remark \ref{NormalFormMod}, the automorphism $\theta$ can be 
written as a product of a conjugation with elementary automorphisms $\tau_{e_1} \ldots \tau_{e_p} \sigma_{u_1} \ldots \sigma_{u_q}$ associated to $T$ where $\sigma_{u_j}$ is a vertex automorphism supported on $v_j$, and $\tau_{e_i}$ is a Dehn twist about the edge $e_i$.
All the elementary automorphisms with support distinct from $v_j$ restrict to conjugations on 
$S_j$, so $\theta (\phi_j(b_0^g))$ and $\theta (c_j)$ are conjugates of $\sigma_{v_j}(b_j)$ and $\sigma_{v_j}(c_j)$ respectively.

Now $\theta \in \Mod_{A\bar{c}}(G)$ so as noted above, it preserves the conjugacy class of $c_j$: 
hence, so does $\sigma_{v_j}$. Let $\sigma = \phi^{-1}_j \circ \sigma_{v_j} \circ \phi_j$: 
the automorphism $\sigma$ preserves the conjugacy class of $c_0$.

Similarly, $\theta'$ fixes $c_k$ so in its normal form 
$\theta' = \Conj(g) \tau'_{e'_1} \ldots \tau'_{e'_{p'}} \sigma'_{u'_1} \ldots \sigma'_{u'_{q'}}$,  
the factor $\sigma'_{v_k}$ is such that $\sigma' = \phi^{-1}_k \circ \sigma'_{v_k} \circ \phi_k$ 
preserves the conjugacy class of $c_0$.

Now $\rho_n \theta(\phi_j(b_0^g))$ is conjugate to 
$$\rho_n \theta (b_j) = \rho_n \sigma_{v_j}(b_j) = \rho_n \sigma_{v_j} \phi_j (b_0) = \rho_n \phi_j \sigma(b_0)$$, 
which is itself equal to $\rho_n^{v_j} \phi_j \sigma(b_0)=\phi_j \rho^{v}_n \sigma(b_0)$. 
Similarly $\rho_m \theta'(\phi_k(b_0^g))$ is conjugate to $\phi_k \rho^{v}_m \sigma'(b_0)$.

Thus (\ref{ConjRelation}) implies that $\phi_j \rho^{v}_n \sigma(b_0)$ is conjugate to $\phi_k \rho^{v}_m \sigma'(b_0)$.

Now these are elements representing a non boundary parallel simple closed curves in $S_j$ and in $S_k$ respectively. 
Thus $S_j$ and $S_k$ are conjugate, so we must have $j=k$. Hence $\phi_j \rho^{v}_m\sigma(b_0)$ is 
conjugate to $\phi_j \rho^{v}_m \sigma'(b_0)$ in $S_j$, which 
contradicts our choice of $\rho^v_n$. 
\end{proof}

\section{Examples and further remarks} \label{ExRem}
We start by giving some simple examples of forking independence between tuples in non abelian free groups.

\begin{ex}
\begin{itemize}
 \item[(i)] Let $\bar{\gamma}_1\in\langle e_1,e_2\rangle$ and $\bar{\gamma}_2\in\langle e_3,e_4\rangle$. Then 
 $\bar{\gamma}_1$ is independent from $\bar{\gamma}_2$ over $\langle [e_1,e_2], [e_3,e_4] \rangle$ in $\F_4$ (see Figure $1$).
 \item[(ii)] Let $\bar{\gamma}_1,\bar{\gamma}_2\in \F_2\setminus\langle[e_1,e_2]\rangle$. Then $\bar{\gamma}_1$ forks with 
 $\bar{\gamma}_2$ over $[e_1,e_2]$ (see Figure $6$).
\end{itemize}

\end{ex}

\begin{figure}[ht!] \label{Fig2}

\centering
\includegraphics[width=.5\textwidth]{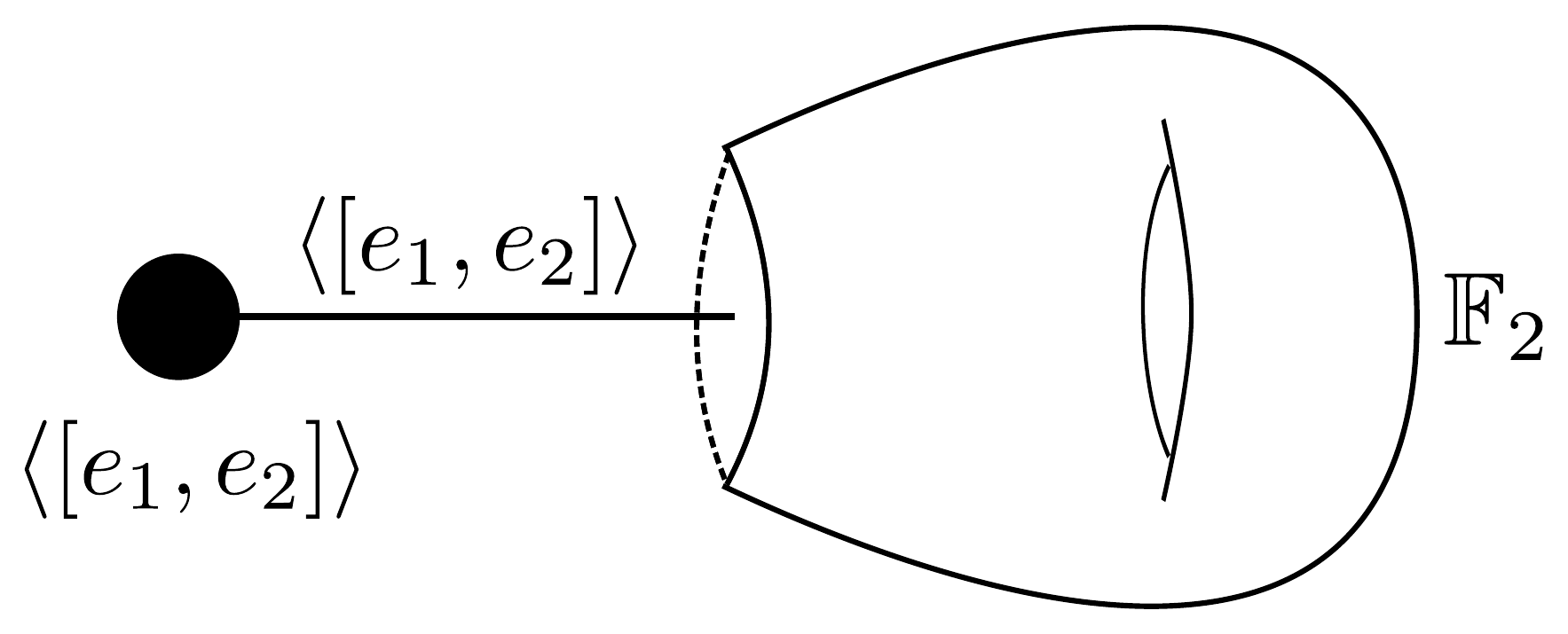}
\caption{A graph of groups decomposition of the pointed tree of cylinders of 
$\F_2$ relative to $H_k = \langle [e_1,e_2]^k\rangle$ for any $k\neq 0$.}

\end{figure}

We moreover note that our results give a complete description of forking independence (for any two tuples in $\F_2$)
over any set of parameters in $\F_2$. The reason is  
that for any set of parameters $A\subseteq \F_2$, either $\F_2$ is freely indecomposable with respect to $A$, or 
$acl(A)$ is a free factor of $\F_2$. We would like to connect this observation with the following question 
we heard from K. Tent.

\begin{question}
Is it possible to prove that $T_{fg}$ is stable using the geometric/algebraic description of forking independence? 
\end{question}

Of course, following our discussion after Fact \ref{forkprop} one needs to find an independence relation (satisfying the properties of Fact \ref{forkprop}) 
in a saturated ``enough'' model of $T_{fg}$, still the intuition comming from $\F_2$ might be useful.

One of the difficulties in characterizing forking (between tuples of elements) in a given torsion-free hyperbolic group $G$, or indeed
in any structure which is not saturated, is that the sequence of tuples 
$(\bar{c}_i)_{i<\omega}$ witnessing the forking of a formula $\phi(x,\bar{c})$ with parameters in $G$ does not have to belong to $G$, 
but in general lies in a saturated elementary extension. 

Thus, it is possible that one needs to move to and understand automorphisms 
of saturated elementary extensions (which are known 
to be non finitely generated). But in the case of torsion-free hyperbolic groups, we are far from 
understanding non-finitely generated models, let alone their automorphism groups.




\bibliography{biblio}

\end{document}